\documentclass[11pt,a4paper]{amsart}
\usepackage[all]{xy}
\usepackage{amssymb}

\numberwithin{equation}{section}
\newcommand{\ie}{{\em i.e.}\ }

\newcommand{\ol}[1]{\overline{#1}}
\newcommand{\wt}[1]{\widetilde{#1}}

\newtheorem{theorem}{Theorem}[section]
\newtheorem*{theorem*}{Theorem}

\newtheorem{lemma}[theorem]{Lemma}
\newtheorem{proposition}[theorem]{Proposition}
\newtheorem{corollary}[theorem]{Corollary}

\newtheorem*{conjecture*}{Conjecture}

\newtheorem{remark}[theorem]{Remark}

\newcommand{\opname}[1]{\operatorname{\mathsf{#1}}}

\renewcommand{\mod}{\opname{mod}\nolimits}

\newcommand{\Mod}{\opname{Mod}\nolimits}

\newcommand{\per}{\opname{per}\nolimits}

\newcommand{\op}{^{op}}
\newcommand{\der}{\cd}

\newcommand{\dimv}{\underline{\dim}\,}

\newcommand{\rank}{\opname{rank}\nolimits}

\newcommand{\ind}{\opname{ind}}

\newcommand{\tria}{\opname{tria}}

\newcommand{\im}{\opname{im}\nolimits}

\newcommand{\tot}{\opname{Tot}\nolimits}

\newcommand{\rad}{\opname{rad}\nolimits}

\newcommand{\Z}{\mathbb{Z}}

\newcommand{\Q}{\mathbb{Q}}
\newcommand{\C}{\mathbb{C}}

\newcommand{\id}{\mathbf{1}}

\newcommand{\Hom}{\opname{Hom}}
\newcommand{\go}{\opname{G_0}}
\newcommand{\RHom}{\opname{RHom}}

\newcommand{\Ext}{\opname{Ext}}

\newcommand{\Aut}{\opname{Aut}}
\newcommand{\End}{\opname{End}}

\newcommand{\ten}{\otimes}
\newcommand{\lten}{\overset{\boldmath{L}}{\ten}}

\newcommand{\Dif}{\opname{Dif}}
\newcommand{\lt}{\opname{L}T}
\newcommand{\rh}{\opname{R}H}
\newcommand{\ca}{{\mathcal A}}
\newcommand{\cb}{{\mathcal B}}
\newcommand{\cc}{{\mathcal C}}
\newcommand{\cd}{{\mathcal D}}

\newcommand{\ch}{{\mathcal H}}

\newcommand{\ck}{{\mathcal K}}
\newcommand{\cl}{{\mathcal L}}
\newcommand{\cm}{{\mathcal M}}
\newcommand{\cn}{{\mathcal N}}

\newcommand{\cp}{{\mathcal P}}
\newcommand{\cR}{{\mathcal R}}

\newcommand{\cs}{{\mathcal S}}
\newcommand{\ct}{{\mathcal T}}

\newcommand{\mh}{\mathfrak{h}}
\newcommand{\mg}{\mathfrak{g}}
\newcommand{\mn}{\mathfrak{n}}

\begin{document}
\author{Changjian Fu}
\title{On root categories of finite-dimensional algebras}

\date{September 2009. Last modified on \today}
\address{Department of Mathematics\\
SiChuan University\\
610064 Chengdu\\
P.R.China
}
\email{
\begin{minipage}[t]{5cm}
changjianfu@scu.edu.cn
\end{minipage}}

\begin{abstract}
For any finite-dimensional algebra $A$ over a field $k$ with finite
 global dimension, we investigate the root category $\cR_A$ as the triangulated hull  of the
 $2$-periodic orbit category of $A$ via the construction of B. Keller in "On triangulated orbit
 categories". This is motivated by Ringel-Hall Lie algebras
 associated to $2$-periodic triangulated categories.
 As an application, we study the Ringel-Hall Lie algebras  for a class of finite-dimensional
  $k$-algebras with global dimension $2$, which turn out to give an
  alternative answer for a question of GIM Lie algebras by Slodowy in "Beyond Kac-Moody algebra,
and inside".
\end{abstract}

 \maketitle

\section{Introduction}
 Root category was first introduced by D. Happel~\cite{Happel1987} for finite-dimensional
hereditary algebra, which was used to characterize a bijection
between the indecomposable objects of the root category for the path
algebra of Dynkin type and the root system of corresponding complex simple
Lie algebra.

Let $A$ be a finite-dimensional hereditary algebra over a field $k$.
Let $\der^b(\mod A)$ be the derived category of finitely generated
right $A$-modules. Then the root category $\cR_A$ of $A$ is defined
to be the $2$-periodic orbit category $\der^b(\mod A)/\Sigma^2$,
where $\Sigma$ is the suspension functor. It was proved  by
Peng-Xiao~\cite{Peng-Xiao1997}, the root category $\cR_A$ is
triangulated via the
 homotopy category of $2$-periodic complexes category of
 $A$-modules. With this triangle structure, Peng and
 Xiao~\cite{Peng-Xiao2000} constructed a so called  Ringel-Hall Lie algebra
 associated to each root category and realized all the symmetrizable
 Kac-Moody Lie algebras. In fact, Peng-Xiao's construction is valid for
 any $\Hom$-finite $2$-periodic triangulated category. In ~\cite{Lin-Peng2005},
 Lin-Peng realized the elliptic Lie algebras of type $D_4^{(1,1)},
 E_6^{(1,1)},E_7^{(1,1)},E_8^{(1,1)}$ via the $2$-periodic orbit categories
 (which are triangulated) of corresponding tubular
 algebras. However, in general, for arbitrary finite-dimensional $k$-algebra
$A$, the $2$-periodic orbit category $\der^b(\mod A)/\Sigma^2$ is
not triangulated with the inherited triangle structure from
$\der^b(\mod A)$({\it cf.} section ~\ref{e:minimal} or
~\cite{Keller2005} ).
 Up to now, there are no suitable $\Hom$-finite $2$-periodic
 triangulated categories to realize the other elliptic Lie algebras
 via the Ringel-Hall Lie algebras approach.

Let $A$ be a finite-dimensional $k$-algebra with finite global
dimension. In ~\cite{Xiao-Xu-Zhang2006}, the authors propose to
study the homotopy category $\ck_2(\cp)$ of
 $2$-periodic complexes category of finitely generated projective
 $A$-modules and give a geometric construction of a Lie algebra over $\mathbb{C}$ directly instead
 of over finite fields in \cite{Peng-Xiao2000}. In this paper, we
 propose to study another $2$-periodic triangulated category $\cR_A$
 called the root category of $A$ via Keller's
 construction~\cite{Keller2005}.  Then Peng-Xiao's construction~\cite{Peng-Xiao2000} gives a Lie
 algebra $\ch(\cR_A)$ for arbitrary finite-dimensional $k$-algebra $A$ with finite global dimension.
  We remark that the root category $\cR_A$ is invariant up to derived equivalence.
  Note that by the $2$-universal property of root category, we
 have an embedding $\cR_A\hookrightarrow \ck_2(\cp)$. When the algebra $A$ is hereditary,
 $\cR_A\cong \ck_2(\cp)\cong \der^b(\mod A)/\Sigma^2$ and  coincides with the original
 definition of Happel.
 We also remark that by using
 $\cR_A$, one can easily construct $2$-periodic triangulated
 categories such that the Grothendieck groups of these categories
 characterize the root lattices for any elliptic Lie algebras ({\it cf.} section~\ref{elliptic
 algebra}). This is one of the motivations to introduce the root category in this paper. The relation between the Ringel-Hall Lie algebras of these
  categories and the corresponding elliptic Lie algebras would be interesting to study in future.

  This paper is organized as follows: in section $2$, for
    any finite-dimensional $k$-algebra
   $A$ of finite global dimension, we introduce the root
   category $\cR_A$ and study its basic properties.
    It is $\Hom$-finite $2$-periodic triangulated category
      and  admits AR-triangles.
     We also give an explicitly characterization of its Grothendieck group.
      In section $3$, we study some motivating examples. In
      particular, we give a minimal example such that the
      $2$-periodic orbit category is not triangulated with the inherited triangle structure.
      In section $4$, we consider the root categories
      of representation-finite hereditary algebras, we show that
      such root categories characterize the algebras up to derived
      equivalence.
       In the last section, we study the Ringel-Hall Lie
      algebras
      of a class of finite-dimensional $k$-algebras with global dimension $2$,
      which turn out to give a negative answer for a question on GIM Lie algebra by
      Slodowy. Let us mention that different counterexamples have been
      discovered in \cite{Alpen1984} by using different approach.
       In the appendix, we discuss the universal property of root category and study
      recollement associated to root categories, which can be use to
      construct various examples inductively such that the $2$-periodic orbit
      category is not triangulated with the inherited triangle
      structure from the bounded derived category.

Throughout this paper, we fix  a field $k$ . All algebras are
finite-dimensional $k$-algebras with finite global dimension. All
modules are  right modules.  Let $\cc$ be a $k$-category. For any
$X,Y\in \cc$, we write $\cc(X,Y)$ for $\Hom_{\cc}(X,Y)$. For a
subcategory $\cm$ in a triangulated category $\ct$, we denote by
$\tria(\cm)$  the thick subcategory of $\ct$ contains $\cm$.

{\bf Acknowledgments.} I deeply thank my supervisor Liangang Peng for
 his guidance and generous patience. Many thanks go to Bernhard Keller for kindly answering to my various questions and for his encouragement.
 I would also like to thank Dong Yang for interesting and useful comments.

\section{root categories for finite-dimensional algebras}
\subsection{2-periodic orbit categories} Let $A$ be a finite-dimensional
$k$-algebra of finite global dimension. Let $\der^b(\mod A)$ be the
bounded derived category of finitely generated $A$-modules and
$\Sigma$ the suspension functor. Consider the left total derived
functor of $A\otimes_kA\op$-module $\Sigma^2 A$
\[\Sigma^2=?\lten_A\Sigma^2 A : \der^b(\mod A)\to \der^b(\mod A),
\]
which is an equivalence. For all $L,M$ in $\der^b(\mod A)$, the
group
\[\der^b(\mod A)(L,\Sigma^{2n}M)
\]
vanishes for all but finitely many $n\in \Z$. The {\it 2-periodic
orbit category}
\[\der^b(\mod A)/\Sigma^{2}
\]
of $A$ is defined as follows:
\begin{itemize}
\item[$\circ$] the objects are the same as those of $\der^b(\mod A)$;
\item[$\circ$] if $L$ and $M$ are in $\der^b(\mod A)$ the space of
morphisms is isomorphic to the space
\[\bigoplus_{n\in \Z}\der^b(\mod A)(L, \Sigma^{2n}M).
\]
\end{itemize}
The composition of morphisms is obviously. Suppose that $A$ is
hereditary.  Then the orbit category is called {\it root category}
of $A$ which  was first introduced by D. Happel in
\cite{Happel1987}.

A $\Hom$-finite $k$-additive triangulated category $\der$ is called
{\it  $2$-periodic triangulated} if:
\begin{itemize}
\item[$\circ$] $\Sigma^2\cong \id$, where $\Sigma$ is the suspension
functor of $\der$;
\item[$\circ$] the endomorphism ring $\End(X)$ for
any indecomposable object $X$ is a finite-dimensional local
$k$-algebra.
\end{itemize}
In particular, the $2$-periodic orbit category of a hereditary
algebra $A$ is $2$-periodic triangulated with canonical triangle
structure proved by Peng-Xiao\cite{Peng-Xiao1997}. However, this is
not true in general. The first example is due to A.Neeman who
considers the algebra $A$ of dual numbers $k[X]/(X^2)$. Then the
$2$-periodic orbit category of $A$ is not triangulated ({\it cf.}
section $3$ of ~\cite{Keller2005}). No example of  algebra with
finite global dimension seems to be known.
\subsection{Root category via Keller's construction}\label{keller's construction}
As shown by Keller in \cite{Keller2005}, if $\der^b(\mod A)$ is
triangulated equivalent to the bounded derived category of a
hereditary category, then the orbit category $\der^b(\mod
A)/\Sigma^2$ is triangulated. In general,  the orbit category is not
triangulated. But a {\it triangulated hull} was defined in
~\cite{Keller2005} as the algebraic triangulated category $\cR_A$
with the following universal properties:
\begin{itemize}
\item[$\circ$] There exists an algebraic triangulated functor $\pi:\der^b(\mod A)\to
\cR_A$;
\item[$\circ$] Let $\cb$ be a dg category and $X$ an object of $\der(A^{op}\otimes
\cb)$. If there exists an isomorphism in $\der(A^{op}\otimes \cb)$
between $\Sigma^{2}A\lten_AX$ and $X$, then the triangulated
algebraic functor $?\lten_AX : \der^b(\mod A)\to \der(\cb)$
factorizes through $\pi$.
\end{itemize}

 Consider $A$ as a dg
algebra concentrated in degree $0$. Let $\cs$ be the dg algebra with
underlying complex $A\oplus \Sigma A$, where the multiplication is
that of the trivial extension:
\[(a,b)(a', b')=(aa',ab'+ba').
\]
Let $\der(\cs)$ be the derived category of $\cs$ and $\der^b(\cs)$
the bounded derived category, {\it i.e.} the full subcategory of
$\der(\cs)$ formed by the dg modules whose homolgy has finite total
dimesion over $k$. Let $\per(\cs)$ be the perfect derived category
of $\cs$, {\it i.e.} the smallest subcategory of $\der(\cs)$
contains $\cs$ and stable under shift,extensions and passage to
direct factors. Clearly, the perfect derived category $\per(\cs)$ is
contained in $\der^b(\cs)$. Denote by $p: \cs\to A$ the canonical
projection. It induces a triangle functor $p_*: \der^b(\mod
A)\to \der^b(\cs)$. By composition we obtain a functor
\[\pi_A: \der^b(\mod A)\to \der^b(\cs)\to \der^b(\cs)/\per(\cs).
\]
Let $\tria(p_*A)$ be the thick subcategory of $\der^b(\cs)$
generated by the image of $p_*A$. By Theorem 2 of \cite{Keller2005},
the triangulated hull of the orbit category $\der^b(\mod
A)/\Sigma^{2}$ is the category
\[\cR_A=\tria(p_*(A))/\per(\cs).
\]
Moreover, there is an embedding $i: \der^b(\mod A)/\Sigma^{2}\hookrightarrow
<A>_{\cs}/\per(\cs)$. If $i$ is dense, then we say that the $2$-periodic orbit category $\der^b(\mod A)/\Sigma^2$ is triangulated with inherited triangle structure of $\der^b(\mod A)$.
 If $A$ is an hereditary algebra, the embedding $i$ is essentially an equivalence of triangulated categories
\[\der^b(\mod A)/\Sigma^{2}\cong \tria(p_*(A))/\per(\cs).
\]

   Since $\cs$ is a negative dg algebra.  It is
well-known that there is a canonical $t$-structure $(\der^{\leq},
\der^{\geq})$  induced by homology over $\der(\cs)$. In particular,
$\der^{\leq}$ is the full subcategory of $\der(\cs)$ whose objects
are the dg modules $X$ such that the homology groups $H^p(X)$
vanishes for all $p>0$. Obviously, the $t$-structure restricts to
the subcategory $\der^b(\cs)$ of $\der(\cs)$. It is not hard to see
that $\der^b(\cs)=\tria(p_*A)$. Then the root category
$\cR_A=\der^b(\cs)/\per S$ in this case. In the following, we call
the triangulated hull $\cR_A$  the {\it root category} of $A$ and
$\pi_A:\der^b(\mod A)\to \der^b(\cs)/\per(\cs)=\cR_A$ the canonical
functor.
\begin{remark}
 One can also consider the construction for the orbit category $\der^b(\mod
 A)/\Sigma^{-2}$ which in fact the same as $\der^b(\mod A)/\Sigma^2$.
 Then one replaces  the dg algebra $\cs$  by  $\cs'=A\oplus
 \Sigma^{-3}A$. The root category defined as $\cR_A=\tria(p_*A)/\per
 \cs'$.
\end{remark}

\subsection{Alternative description of $\cR_A$}\label{s:dg orbit category}
There is another description of $\cR_A$ in \cite{Keller2005}. Let $\ca$ be the dg category of bounded complexes of finitely generated projective $A$-modules. Naturally, the tensor product of $\Sigma^2 A$ define a dg functor from $\ca$ to $\ca$. Then one can form the {\it dg orbit category} $\cb$ as the dg category with the same objects of $\ca$ and such that  for any $X, Y\in \cb$, we have
\[\cb(X,Y)\cong  \bigoplus_{n \in \Z}\ca(X, \Sigma^{2n}Y).
\]
 Now we have an equivalence of categories
\[\der^b(\mod A)/\Sigma^2\cong \ch^0(\cb).
\]
 Let $\der(\cb)$ be the derived category of the dg category $\cb$.
 Let the ambient triangulated category $\cm$ be the triangulated subcategory of
 $\der(\cb)$ generated by the representable functors. Then theorem 2 of \cite{Keller2005} implies that $\cR_A\cong \cm$.

 \begin{proposition}\label{2-periodic}
 Let $A$ be a finite-dimensional $k$-algebra of finite global dimension.
 Then the root category $\cR_A$ is a $\Hom$-finite $2$-periodic triangulated category.
 \end{proposition}
 \begin{proof}
 The $\Hom$-finiteness follows from the description of $\cm$,
 since the homomorphisms between representable functors of $\cb$ are finite-dimensional over $k$.
  Consider the $\cb\otimes \cb\op$-module $X:$ $X(A, B)=\cb(A, B)$ for any $A, B\in \cb$, it induces the identity functor
 \[\id: \der(\cb)\to \der(\cb).
 \]
 One can also consider the $\cb\otimes \cb\op$-module $Y(A, B)=\Sigma^2\cb(A,B)$ for any $A, B\in \cb$.
  Clearly, the module $Y$ induces the triangle functor
 \[\Sigma^2: \der(\cb)\to \der(\cb).
 \]
  By the definition of the dg orbit category $\cb$, we know that $X$ is isomorphic to $Y$ as
  $\cb\otimes \cb\op$-modules, which will induce an invertible morphism
  $\eta: \id\to \Sigma^2$ by Lemma 6.1 of \cite{Keller1994}.
  Thus, to show that $\cR_A$ is $2$-periodic triangulated category,
  it suffices to show that $\cR_A$ is Krull-Schmidt category.
  It suffices to prove that each idempotent morphism of $\cR_A$ is split, {\it i.e.} $\cR_A$ is idempotent completed.
  Recall that $\cR_A=\cm\subset \der(\cb)$ and $\der(\cb)$ is
  idempotent complete since $\der(\cb)$ admits arbitrary direct
  sums. Moreover, $\cm$ is closed under direct summands in
  $\der(\cb)$. Now the result follows from the well-known fact that
    if an additive category $\cc$ is idempotent completed,
    then a full subcategory $\der$ of $\cc$ is idempotent completed if
    and only if $\der$ is closed under direct summands.
 \end{proof}

\subsection{Serre functor over $\cR_A$}
Keep the notations above. Let $D=\Hom_k(?, k)$ be the usual duality
over $k$. The
 $\cs\otimes_k\cs\op$-module $D\cs$ induces a triangle functor
 \[?\lten_{\cs}D\cs:\der(\cs)\to \der(\cs).
 \]
  We have the following well-known fact(see {\it e.g.} lemma 1.2.1 of
\cite{Amiotthesis}).
\begin{lemma}\label{l:non-degenerate}
There is a non-degenerate bilinear form
\[\alpha_{X,Y}:\der(\cs)(X, Y)\times \der(\cs)(Y, X\lten_{\cs}D\cs)\to k
\]
which is bifunctorial for $X\in \per(\cs)$ and $Y\in \der^b(\cs)$.
\end{lemma}
\begin{proposition}\label{p:auto-equivalence}
The functor $?\lten_{\cs}D\cs$ restricts to auto-equivalences
\[?\lten_{\cs}D\cs: \der^b(\cs)\to \der^b(\cs), ?\lten_{\cs}
D\cs:\per(\cs)\to \per(\cs).
\]
\end{proposition}
\begin{proof}
Since $A$ has finite global dimension, we know that $DA\in \per A$.
One can easily deduce that $D\cs\in \per \cs$. Similarly, we have
$\cs\in \tria(D\cs)\subseteq \der(\cs)$.  This particular implies
that $D\cs$ is a small generator of $\der(\cs)$. It is not hard to
show that
\[\der(\cs)(\cs, \Sigma^n\cs)\cong \der(\cs)(D\cs,
\Sigma^nD\cs), n\in \Z.
\]
Thus by Lemma 4.2 of \cite{Keller1994}, we infer that
$?\lten_{\cs}D\cs$ is an equivalence over $\der(\cs)$. Now the
functor $?\lten_{\cs}D\cs$ restricts to $\per(\cs)$ follows from
$\tria(D\cs)=\per(\cs)$. Similarly, recall that we have
$\tria(p_*A)=\der^b(\cs)$ and $A\lten_{\cs}D\cs\cong
\Sigma^{-1}DA\in \der^b(\cs)$. Now again by the finite global
dimension of $A$, we have $A\in \tria(p_*(DA))\subseteq
\der^b(\cs)$. In particular, we have $\tria(p_*(DA))=\der^b(\cs)$.
Thus, $?\lten_{\cs}D\cs$ restricts to an equivalence
$?\lten_{\cs}D\cs:\der^b(\cs)\to \der^b(\cs)$.
\end{proof}

Before going to state the next result, we recall Amiot's
construction~\cite{Amiot2008} of bilinear form for quotient
category. Let $\ct$ be a triangulated category and $\cn\subset \ct$
a thick subcategory of $\ct$. Assume $\nu$ is an auto-equivalence of
$\ct$ such that $\nu(\cn)\subset \cn$. Moreover we assume that there
is a non degenerate  bilinear form:
\[\beta_{N,X}:\ct(N,X)\times \ct(X,\nu N)\to k
\]
which is bifunctorial in $N\in \cn$ and $X\in \ct$.  Let $X,Y\in
\ct$. A morphism $p: N\to X$ is called a {\it local $\cn$-cover of
$X$ relative to $Y$} if $N$ is in $\cn$ and it induces an exact
sequence:
\[0\to \ct(X,Y)\xrightarrow{p^*} \ct(N,Y).
\]
The following theorem is due to Amiot~\cite{Amiot2008}.
\begin{theorem}\label{t:Amiot}
\begin{itemize}
\item[1)] The bilinear form $\beta$ naturally induces a bilinear
form:
\[\beta'_{X,Y}:\ct/\cn(X,Y)\times \ct/\cn
(Y,\nu \Sigma^{-1}X)\to k\] which is also bifunctorial for $X,Y\in
\ct/\cn$;
\item[2)] Assume further $\ct$ is $\Hom$-finite. If there exists a
local $\cn$-cover of $X$ relative to $Y$ and a local $\cn$-cover of
$\nu Y$ relative to $X$, then the bilinear form $\beta'_{X,Y}$ is
non-degenerate.
\end{itemize}
\end{theorem}

Recall that $\cR_A=\der^b(\cs)/\per(\cs)$.  Now we have the
following
\begin{proposition}\label{non-degenerate}
\begin{itemize}
\item[1)] The bilinear form $\alpha$ induces a bifunctorial bilinear
form $\alpha'$:
\[\alpha'_{X,Y}:\cR_A(X, Y)\times
\cR_A(Y,\Sigma^{-1}X\lten_{\cs}D\cs)\to k;
\]
\item[2)] The bilinear form $\alpha'$ is non-degenerate over
$\cR_A$.
\end{itemize}
\end{proposition}
\begin{proof}
The first statement  follows form lemma~\ref{l:non-degenerate},
proposition~\ref{p:auto-equivalence} and theorem~\ref{t:Amiot}
directly.

Let $P_A=\tot(\cdots\to \Sigma^{n}\cs\to\Sigma^{n-1}\cs\to \cdots\to
\Sigma^2\cs\to \Sigma \cs\to \cs\to 0\to \cdots)$, {\it i.e.} $P_A$
is the projective resolution of $\cs$-module $A$. Then one can
easily see that $\der^b(\cs)(A, \Sigma^m A)$ is finite dimension
over $k$ for any $m\in \Z$. In particular, we have
\[\der^b(\cs)(A, \Sigma^{2m}A)\cong A \ \ \text{and }\
\der^b(\cs)(A, \Sigma^{2m+1}A)=0,
\]
for $m\geq 0$ and $\der^b(\cs)(A,\Sigma^mA)=0$ for $m<0$. Since
$p_*A$ generates the category $\der^b(\cs)$, which implies that
$\der^b(\cs)$ is $\Hom$-finite, {\it i.e.} for any $X, Y\in
\der^b(\cs)$, we have $\dim_k\der^b(\cs)(X, Y)<\infty$. Since
 the non-degeneracy is extension closed,
it suffices to show that $\alpha'_{\Sigma^nA, \Sigma^mA}$ is
non-degenerate. Equivalently, it suffices to show that $\alpha'_{A,
\Sigma^nA}$ is non-degenerate for any $n\in \Z$. By Theorem
~\ref{t:Amiot} $2)$, it suffices to show that there exists a local
$\per \cs$-cover of $A$ relative to $\Sigma^nA$ and a local $\per
\cs$-cover of $\Sigma^nA$ relative to $A\lten_{\cs}D\cs$. For $n<0$,
since $\der^b(\cs)(A,\Sigma^nA)=0$, one can take $p:\cs\to A$ be the
local $\per \cs$ of $A$ relative to $\Sigma^nA$.

Suppose that $n\geq 0$. Let
 \[P_{A,\Sigma^nA}:=\tot(\cdots\to 0\to \Sigma^n\cs\to \Sigma^{n-1}\cs\to \cdots\to \Sigma \cs\to \cs\to 0\to
 \cdots).
\]
Clearly $P_{A,\Sigma^nA}\in \per \cs$. One can easily to see that
$p: P_{A,\Sigma^nA}\to A$ is a local $\per \cs$-cover of $A$
relative to $\Sigma^nA$. Note that $A\lten_{\cs}D\cs\cong
\Sigma^{-1}DA$. A local $\per \cs$-cover of $\Sigma^n A$ relative to
$\Sigma^{-1}DA$ is equivalent to a local $\per \cs$-cover of $A$
relative to $\Sigma^{-n-1}DA$. If $n\geq 0$, one can easily show
that $\der^b(\cs)(A, \Sigma^{-n-1}DA)=0$. Suppose that $n< 0$. One
can show that
\[P_{A,\Sigma^{-n-1}DA}:=\tot(\cdots\to 0\to \Sigma^{-n-1}\cs\to \Sigma^{-n-2}\cs\to \cdots\to \Sigma \cs\to \cs\to 0\to
 \cdots)\to A
\]
is a local $\per \cs$-cover of $A$ relative to $\Sigma^{-n-1}DA$.
\end{proof}
\begin{theorem}\label{t:Serre functor}
The root category $\cR_A$ admits Auslander-Reiten triangles.
\end{theorem}
\begin{proof}
By proposition \ref{non-degenerate}, we know that
$\Sigma^{-1}?\lten_{\cs}D\cs$ is the Serre functor of
$\cR_A=\der^b(\cs)/\per \cs$. Now the result follows form Reiten-Van
den Bergh's result in \cite{Reiten-Van den Bergh2002}.
\end{proof}

\subsection{The Grothendieck group of $\cR_A$}\label{s:grothendieck
group}
Suppose that the algebra $A$ has $n$ non-isomorphic simple modules,
say $S_1, \cdots. S_n$ and $P_1, \cdots, P_n$  the corresponding
projective covers. Since $\cs$ be the trivial extension of $A$ with
non-standard gradation, $S_1, \cdots, S_n$ are also non-isomorphic
simple modules for $\cs$. By the existence of $t$-structure over
$\der^b(\cs)$, it is not hard to see that  the Grothendieck group
$\go(\der^b(\cs))$ of $\der^b(\cs) $ is isomorphic to
$\Z[S_1]+\cdots+\Z[S_n]$. Indeed, consider the inclusion algebra
homomorphism $i: A\to \cs$, which induces a triangle functor $i_*:
\der^b(\cs)\to \der^b(\mod A)$. Compose with the functor
$p_*:\der^b(\mod A)\to \der^b(\cs)$, the result follows from that
$\go(\der^b(\mod A))\cong \Z^n$.

 We have the following exact
sequence of triangulated categories
\[\per \cs\rightarrowtail \der^b(\cs)\twoheadrightarrow \cR_A,
\]
which induces an exact sequence of Grothendieck groups
\[\go(\per \cs)\xrightarrow{i^*} \go(\der^b(\cs))\xrightarrow{\phi}\go(\cR_A)\to 0.
\]
In particular, we have $\go(\cR_A)\cong \go(\der^b(\cs))/\im i_*$.
Let $\wt{P_i}=P_i\oplus \Sigma P_i$. It is not hard to see that
$\wt{P_i}$ are all the indecomposable projective objects in
$\der(\cs)$ up to shifts. Note that $\cs$ is negative, as remark in
\cite{Keller1994}, each compact object is an extension of direct sum
of $\Sigma^n\wt{P_i}, n\in \Z$. In particular, in the Grothendieck
group $\go(\per \cs)$, for any $X\in \per \cs$, $[X]$ is a finite
sum of $[\wt{P_i}], i=1, \cdots, n$.  It is easy to see that
$i^*([\wt{P_i}])=0\in \go(\der^b(\cs))$. Thus, the image of $i^*$ is
zero and the induced linear map $\phi: \go(\der^b(\cs))\to
\go(\cR_A)$ is an isomorphism. Compose $\phi$ with
$p^*:\go(\der^b(\mod A))\to \go(\der^b(\cs))$ induced by
$p_*:\der^b(\mod A)\to \der^b(\cs)$, which is exactly the induced
map $\pi_A^*:\go(\der^b(\mod A))\to \go(\cR_A)$ by the canonical
functor $\pi_A:\der^b(\mod A)\to \cR_A$. In particular, $\pi_A^*$ is
an linear isomorphism.

Define the Euler bilinear form  $\chi_{\cR_A}(-,-)$on $\go(\cR_A)$
by
\[\chi_{\cR_A}([X], [Y])=\dim_k \cR_A(X, Y)-\dim_k\cR_A(X, \Sigma Y)
\]
for any $X, Y\in \cR_A$. Clearly, it is well-defined due to the
$2$-periodic property of $\cR_A$. Let $\chi_A(-,-)$ be the Euler
bilinear form over $\der^b(\mod A)$, {\it i.e.}
\[\chi_A([X],[Y])=\sum_{i\in \Z} (-1)^i\dim_k \der^b(\mod A)(X, \Sigma^i Y)
\]
for any $X, Y\in \der^b(\mod A)$. Since $\go(\cR_A)$ is generated by
$[\pi_AS_i], i=1, \cdots, n$, where $S_i$ are simple $A$-modules, we
have $\chi_{\cR_A}([\pi_AS_i], [\pi_AS_j])=\chi_A([S_i], [S_j])$ for
any $1\leq i,j\leq n$. Thus the symmetric bilinear form $(-|-)$ over
$\go(\cR_A)$  given by
\[([X]|[Y])=\chi([X],[Y])+\chi([Y],[X])
\]
is the same over $\go(\der^b(\mod A))$ via the isomorphism $\pi_A^*$.

In particular, we have proved the following
\begin{proposition}
The canonical functor $\pi_A:\der^b(\mod A)\to \cR_A$ induces an
isometry $\pi_A^*:\go(\der^b(\mod A))\to \go(\cR_A)$, {\it i.e.} a
linear map such that $\chi_{A}(x,y)=\chi_{\cR_A}(\pi_A^*x,\pi_A^*y)$
for any $x,y\in \go(\der^b(\mod A))$.
\end{proposition}

\begin{remark}
If $A$ is finite-dimensional  hereditary  $k$-algebra. We have
$\cR_A\cong \der^b(\mod A)/\Sigma^2$, this shows that
$\go(\der^b(\mod )/\Sigma^2)\cong \go(\der^b(\mod A))$. Let $\pi:
\der^b(\mod A)\to \cR_A$ which is dense in this case. Let
$\ol{\go(\cR_A)}$ be the Grothendieck group of $\cR_A$ induced by
the triangles of image $\pi$. We have $\go(\cR_A)\cong
\ol{\go(\cR_A)}$.
\end{remark}

Let $c_{\cR_A}$ be the  automorphism of $\go(\cR_A)$ induced by the
Auslander-Reiten translation of $\cR_A$. Let $c_{A}$ be the
automorphism of $\go(\der^b(\mod A))$ induced by the AR translation
of $\der^b(\mod A)$.
\begin{proposition}\label{coexter}
 $c_A$ identifies with  $c_{\cR_A}$ via the isomorphism $\pi_A^*: \go(\der^b(\mod A))\to \go(\cR_A)$.
\end{proposition}
\begin{proof}
Let $P_i$, $i=1, \cdots, n$ be the non-isomorphic indecomposable
projective modules of $A$. Since $A$ has finite global dimension, we
know that $[P_i]$, $i=1, \cdots, n$ also form a
 basis of $\go(\der^b(\mod A))$. Via the isomorphism $\pi_A^*$, $[\pi_A P_i]$ is also a basis of $\go(\cR_A)$.
 Thus, it suffices to show that $c_{\cR_A}(\pi_A P_i)$ coincides with $\pi_A^*c_{A}(P_i)$, $i=1, \cdots, n$. By the definition of $c_{\cR_A}$ and $c_A$, we have
\[c_{A}([P_i])=[\Sigma^{-1}P_i\lten_{A}DA], \ c_{\cR_A}([\pi_AP_i])=[\Sigma^{-2} P_i\lten_{\cs}D\cs],
\]
One can easily check that $[\Sigma^{-1}P_i\lten_{A}DA]=[\Sigma^{-2} P_i\lten_{\cs}D\cs]$ in $\go(\cR_A)$.
\end{proof}

\section{Motivating Examples}
\subsection{14 exceptional unimodular singularities}
 Inspired by the theory that the universal
deformation and simultaneous resolution of a simple singularity are
described  by the corresponding simple Lie
algebras~\cite{Brieskorn1970}, K. Saito associated in
\cite{Saito1985}, a generalization of root system to any regular
weight systems \cite{Saito1987}, and asks to construct a suitable
Lie theory in order to reconstruct the primitive forms for  the
singularities. This is well-done for simple singularities and simple
elliptic singularities. But, in general, it is not clear how to
construct a suitable Lie theory even for $14$ exceptional unimodular
singularities.

 Based on the duality theory of
the weight systems \cite{Saito1998} and the homological mirror
symmetry, Kajiura-Saito-Takahashi~\cite{Kajiura-Saito-Takahashi2005}
(Takahashi~\cite{Takahashi2005}) propose to study
 the triangulated
category $HMF_A^{gr}(f_W)$ of matrix factorizations of the
homogenous polynomial $f_W$ associated to a simple singularity $W$,
then the root system appears as the set of  the isomorphism classes
of  the exceptional objects via the Grothendieck group of the
triangulated category. This approach has been generalized to  the
case of regular weight systems with smallest exponent $\epsilon=-1$
in \cite{Kajiura-Saito-Takahashi2007} which includes the 14
exceptional unimodular singularities. In
\cite{Kajiura-Saito-Takahashi2005}, the authors show that the
category $HMF_A^{gr}(f_W)$ is  triangulated equivalent to the
bounded derived category of finitely generated modules over the path
algebra of the corresponding $ADE$-type. In
\cite{Kajiura-Saito-Takahashi2007}, they  show that the category
$HMF_A^{gr}(f_W)$ is triangulated equivalent to the bounded derived
category of finitely generated modules of certain finite-dimensional
algebras $A_W$.  Moreover, the Grothendieck groups of these
triangulated categories characterize the strange duality for the 14
exceptional unimodular singularities. In his survey article
\cite{Saito2008}, K. Saito proposes three methods to construct Lie
algebras for each exceptional singularity and asks which  Lie
algebra satisfies some extra requirements, for more details see
\cite{Saito2008}:
\begin{itemize}
\item[i)] the Lie algebra defined by the Chevalley generators and
generalized Serre relations for the Cartan matrix associated to
algebra $A_W$;
\item[ii)] the Lie subalgebra comes form vertex operator algebra for
the Grothendieck group $K_0(\cd^b(\mod A_W))$;
\item[iii)] the algebra constructed by Ringel-Hall construction for
the derived category of $\cd^b(\mod A_W)$.
\end{itemize}
We remark that for the simple singularities which are self-dual, if
one consider the Lie algebra iii) in the sense of Peng-Xiao
\cite{Peng-Xiao2000} for the root category, then these three Lie
algebras are isomorphic to each other. For the case of
$\epsilon=-1$, we remark that the algebra $A_W$ has global dimension
$2$ and it is not derived equivalent to any hereditary category.
 It is not clear that whether the $2$-periodic orbit
category $\cd^b(\mod A_W)/\Sigma^2$  is triangulated or not. Now the
root category $\cR_{A_W}$ seems to be a suitable consideration for
the Lie algebra $iii)$.  Then Peng-Xiao's Theorem
\cite{Peng-Xiao2000}(see also \cite{XuFan2009} ) implies that there
is a Lie algebra $\ch(\cR_{A_W})$ associated with $\cR_{A_W}$. We
remark that we do not know whether the Grothendieck group of $\cR_A$
is proper or not. In this case, the automorphism $c_{\cR_{A_W}}$ has
finite order $h_W$, where $h_W$ is the order of the Milnor monodromy
of the corresponding singularity $W$.

\subsection{An algebra of global dimension 2 }\label{elliptic algebra}

Let $Q$ be the following quiver
\[\xymatrix{3\ar@/^/[r]^{\alpha_1}\ar@/_/[r]_{\alpha_2}&2\ar@/^/[r]^{\beta_1}\ar@/_/[r]_{\beta_2}&1.}
\]
Let $I$ be the ideal generated by the relations $\beta_i\circ\alpha_i=0$. Let $A=kQ/I$ be the quotient algebra.  The global
dimension of $A$ is $2$. Let $S_i, i=1,2,3$, be the non-isomorphic
simple modules of $A$. Consider the Euler symmetric bilinear form of
$\go(\cR_A)$ given by
\[([X]|[Y])=\chi(X,Y)+\chi(Y,X),
\]
 for any $X, Y\in \cR_A$.
One can check that $([S_1]|[S_2])=-2, ([S_1]|[S_3])=2,
([S_2]|[S_3])=-2$. The bilinear form $(-|-)_{\go(\cR_A)}$ is
degenerate over $\go(\cR_A)$, one extends  $\go(\cR_A)$ to $\cl$
such that the bilinear form $(-|-)_{\cl}$ over $\cl$ is
non-degenerate, {\it i.e.} the restriction
$(-|-)_{\cl}|_{\go(\cR_A)}$ coincides with $(-|-)_{\go(\cR_A)}$.

Let $V_{\cl}$ be the lattice vertex operator algebra associated to
$\cl$. If we consider the Lie algebra $\mathfrak{g}_A$ generated by
the vertex operators $e^{\pm[S_i]}$ in the Lie algebra
$V_{L}/DV_{L}$, where $D$ is the derivative operator, then
$\mathfrak{g}_A$ is isomorphic to  the elliptic
algebra~\cite{Yoshii2000} of type $A_1^{(1,1)}$ and also isomorphic
to the toroidal algebra~\cite{Rao-Moody1994} of $\mathfrak{sl}_2$.

 Now consider the root category $\cR_A$ of $A$,  there is a Lie algebra $\ch(\cR_A)$(Ringel-Hall Lie algebra)
 associated with $\cR_A$. We would
 like to know what is the relation between $\ch(\cR_A)$ and the elliptic
 algebra $A_1^{(1,1)}$? At this moment, we can only show that $u_{[\Sigma S_1]}, u_{[S_2]},
 u_{[S_3]}$ satisfy the GIM Lie algebra relations \cite{Slodowy1986}.

  We also remark
 that up to now, there is not any triangulated category to realize
 the elliptic algebras of type $A$ and $D$(except for $D_4^{(1,1)}$) via the approach of Ringel-Hall
  Lie algebras. Similar to the above example,  by triangular extension of algebras, for any elliptic Lie algebra,
 one can construct a $2$-periodic triangulated category (possibly not unique)
 such that the Grothendieck group with the symmetric Euler bilinear form  characterizes the root lattice for the corresponding elliptic Lie algebra.
\subsection{A minimal example}\label{e:minimal} Let $Q$ be the following quiver with
relation
\[\xymatrix{2\ar@/^/[r]^{\alpha}&1\ar@/^/[l]^{\beta}}
\]
where $\beta\circ \alpha=0$. Let $A$ be the quotient algebra of path
algebra $kQ$ by the ideal generated by $\beta\circ \alpha$. Then $A$
is representation-finite and has global dimension $2$. Let
$\der^b(\mod A)$ be the bounded derived category of finitely
generated right $A$-modules. Let $\ca$ be the dg enhance of
$\der^b(\mod A)$, i.e. the dg category of bounded complexes of
finite generated projective $A$-modules. Let $\Sigma^2:\ca\to \ca$
be the dg enhance of the square of suspension functor of
$\der^b(\mod A)$. Let $\cb$ be the dg orbit category of $\ca$
respect to $\Sigma^2$({\it cf.} section ~\ref{s:dg orbit category}).
The canonical dg functor $\pi:\ca\to \cb$ yields a
$\cb\otimes_k\ca\op$-module
\[(B,A)\to \cb(B,\pi A),
\]
which induce the standard functors
\[\xymatrix{\der(\ca)\ar@<1ex>[r]^{\pi_*}&\der(\cb)\ar@<1ex>[l]^{\pi_{\rho}}}.
\]
Note that $\cR_A$ is the triangulated subcategory of $\der(\cb)$
generated by the representable functors.  We also have a triangle
equivalence $F:\der(\Mod A)\to \der(\ca)$. Now the composition
\[\der^b(\mod A)\hookrightarrow \der(\Mod A)\xrightarrow{F}\der(\ca)\xrightarrow{\pi_*}\der(\cb)
\]
gives the canonical functor $\pi_*:\der^b(\mod A)\to \cR_A$.
\begin{proposition}The canonical function $\pi_*: \der^b(\mod A)\to
\cR_A$ is not dense.
\end{proposition}
\begin{proof}
We will construct an object in $\cR_A$ which is not in the image of
$\pi_*$. Let $S_i, i=1,2$ be the simple $A$-modules associated to
the vertices $i$ and $P_i, i=1,2$ be the corresponding
indecomposable projective modules. Let $l:P_2\to P_1$ be the
embedding and $\gamma: P_1\twoheadrightarrow S_1\hookrightarrow
P_2$. Let X be the complex $\cdots\to 0\to
P_2\xrightarrow{(l,0)}P_1\oplus P_2\xrightarrow{(0,l)^t}P_1\to
0\cdots$, where $P_1\oplus P_2$ is in the $0$-th component. Let $Y$ be
the complex $\cdots\to 0\to 0\to P_2\xrightarrow{0}P_2\to 0\cdots$,
where the left $P_2$ is  in the $0$-th component. Let $f\in
\Hom_{\der^b(\mod A)}(X,Y)$ and $g\in \Hom_{\der^b(\mod A)}(X,
\Sigma^2 Y)$ be the followings
\[\xymatrix{0\ar[r]&P_2\ar[d]^{0}\ar[r]^{(l,0)}&P_1\oplus P_2\ar[d]^{(\gamma, 1)^t}\ar[r]^{(0,l)^t}&P_1\ar[d]^{\gamma}\ar[r]&0\\
0\ar[r] &0\ar[r]&P_2\ar[r]^{0}&P_2\ar[r]&0}
\]
\[\xymatrix{0\ar[r]\ar[d]^{0}&P_2\ar[d]^{1}\ar[r]^{(l,0)}&P_1\oplus P_2\ar[d]^0\ar[r]^{(0,l)^t}&P_1\ar[r]&0\\
P_2\ar[r]^0&P_2\ar[r]&0}
\]
Consider the mapping cone of $\pi_*(f+g)$ in $\cR_A$, we claim that
the mapping cone of $\pi_*(f+g)$ is not in the image of $\pi_*$.
Consider the triangle
\[\pi_*(X)\xrightarrow{\pi_*(f+g)}\pi_*(Y)\to Z\to \Sigma \pi_*(X).
\]
Applying the functor $\pi_{\rho}$, we get a triangle in $\der(\Mod
A)$
\[\pi_{\rho}\pi_*(X)\xrightarrow{\pi_{\rho}\pi_*(f+g)}\pi_{\rho}\pi_*(Y)\to\pi_{\rho}Z\to \Sigma \pi_{\rho\pi_*( X)}
\]
Note that for any $X\in \der^b{\mod A}$, we have
$\pi_{\rho}\pi_*(X)\cong \oplus_{i\in\Z}\Sigma^{2i}X$. Thus,
$\pi_{\rho}Z$ is isomorphic to the mapping cone of the following
chain map of complexes
\[\xymatrix{\cdots\ar[r]&P_1\oplus P_2\ar[d]^{(\gamma, 1)^t}\ar[r]^{\left(\begin{array}{cc}0&0\\l&0\end{array}\right)}&
P_1\oplus P_2\ar[d]^{(\gamma, 1)^t}\ar[r]^{\left(\begin{array}{cc}0&0\\l&0\end{array}\right)} &P_1\oplus P_2\ar[d]^{(\gamma,1)^t}\ar[r]^{\left(\begin{array}{cc}0&0\\l&0\end{array}\right)}&P_1\oplus P_2\ar[d]^{(\gamma,1)^t}\ar[r]&\cdots\\
\cdots\ar[r]&P_2\ar[r]^0&P_2\ar[r]^0&P_2\ar[r]^0&P_2\ar[r]&\cdots}
\]
In particular,  the mapping cone is
\[\xymatrix{\cdots\ar[r] &P_2\oplus P_1\oplus P_2\ar[r]^{\left(\begin{array}{ccc}0 &0& 0\\-\gamma&0&0\\-1&-l&0\end{array}\right)}&
P_2\oplus P_1\oplus P_2\ar[r]^{\left(\begin{array}{ccc}0 &0&
0\\-\gamma&0&0\\-1&-l&0\end{array}\right)}&P_2\oplus P_1\oplus
P_2\ar[r]&\cdots}
\]
Let $h:P_1\twoheadrightarrow S_1\hookrightarrow P_1$, consider the
complex $P:\cdots\to P_1\xrightarrow{h}P_1\xrightarrow{h}P_1\to
\cdots$, one can check that
\[\xymatrix{\cdots\ar[r] &P_2\oplus P_1\oplus P_2\ar[d]^{(-l,1,0)^t}\ar[r]^{\left(\begin{array}{ccc}0 &0& 0\\-\gamma&0&0\\-1&-l&0\end{array}\right)}&
P_2\oplus P_1\oplus
P_2\ar[d]^{(-l,1,0)^t}\ar[r]^{\left(\begin{array}{ccc}0 &0&
0\\-\gamma&0&0\\-1&-l&0\end{array}\right)}&P_2\oplus P_1\oplus
P_2\ar[d]^{(-l,1,0)^t}\ar[r]&\cdots\\
\cdots\ar[r]&P_1\ar[r]^h&P_1\ar[r]^h&P_1\ar[r]^h&\cdots}
\]
 is a quasi-isomorphism. In particular, $\pi_{\rho} Z$ is
 isomorphic to $P$ in $\der(\Mod A)$. If there exists $U\in \der^b(\mod A)$ such that
 $\pi_*(U)=Z$, then $\pi_{\rho}Z\cong \oplus_{i\in Z}\Sigma^{2i}U$.
 But one can easily show that $P$ is indecomposable in $\der(\Mod
 A)$. This completes the proof.
\end{proof}
The above example implies that in general the orbit category
$\der^b(\mod A)/\Sigma^2$ is not triangulated even with small global
dimension. In the appendix, we propose a way to construct various
examples from a known one by using recollement associates to root
categories. It would be interesting to konw that whether one can give an
example without oriented cycles such that the orbit category
$\der^b(\mod A)/\Sigma^2$ is not triangulated with inherited triangle structure.

\section{The ADE root categories}
In this section, we focus on the root categories of
finite-dimensional hereditary algebras of Dynkin type. We will show
that such root categories characterize the algebras up to derived
equivalence.
\subsection{Separation of AR-components}
Let $A$ be a finite dimensional $k$-algebra with finite global
dimension. Let $\pi_A:\der^b(\mod A)\to \cR_A$ be the canonical
triangle functor. By theorem~\ref{t:Serre functor}, we know that
$\cR_A$ has Serre functor, equivalently, Auslander-Reiten
triangles(AR-triangles). When $\pi_A$ is dense, it is quite easy to
show that $\pi_A$ preserves the AR-triangles, {\it i.e.} each
AR-triangle of $\cR_A$ comes from an AR-triangle of $\der^b(\mod A)$
via the canonical functor $\pi_A$. In particular, the AR-quiver of $\der^b(\mod A)$
determines the AR-quiver of $\cR_A$. In general, we have the
following
\begin{theorem}\label{t:separated}
Let $A$ be a finite-dimensional $k$-algebra with finite global
dimension and $\pi_A:\der^b(\mod A)\to \cR_A$ the canonical functor.
Then the functor $\pi_A$ maps AR-triangles of $\der^b(\mod A)$ to
AR-triangles of $\cR_A$. As a consequence, there is no irreducible
morphism between $\im \pi_A$ and $\cR_A\backslash \im \pi_A$.
\end{theorem}
\begin{proof}
Recall that for arbitrary objects $X, Y\in \der^b(\mod A)$, we have
canonical isomorphism
\[\cR_A(\pi_A(X), \pi_A(Y))\cong \bigoplus_{i\in \Z}\der^b(\mod
A)(\Sigma^{2i}X, Y),
\]
and $\der^b(\mod A)(\Sigma^{2i}X, Y)$ vanishes for all but finitely
many $i$. Let $S$ and $\wt{S}$ be the Serre functors of $\der^b(\mod
A)$ and $\cR_A$ respectively. Firstly, we show that $\pi_AS(X)\cong
\wt{S}\pi_A(X)$ for any  indecomposable object $X\in \der^b(\mod
A)$.   Consider the  functor $D\cR_A(?,\pi_AS(X))$ over $\cR_A$,
where $D=\Hom_k(?,k)$ is the usual duality of $k$. We have the
following canonical isomorphism
\begin{eqnarray*}
D\cR_A(\pi_AX, \pi_AS(X))&\cong& D(\bigoplus_{i\in
\Z}\der^b(\mod A)(\Sigma^{2i}X, S(X)))\\
&\cong &\bigoplus_{i\in \Z}D\der^b(\mod A)(\Sigma^{2i}X,
S(X))\\
&\cong&\bigoplus_{i\in \Z}\der^b(\mod A)(X, \Sigma^{2i}X)\\
&\cong&\cR_A(\pi_AX, \pi_AX)
\end{eqnarray*}
The indecomposable property implies that $\cR_A(\pi_AX, \pi_AX)$ is
a local $k$-algebra. Let $\eta\in D\cR_A(\pi_AX, \pi_AS(X))$ be the
image of $1_{\pi_AX}\in \cR_A(\pi_AX, \pi_AX)$ via the canonical
isomorphism. Let $\eta^*:\cR_A(\pi_AX,?)\to D\cR_A(?,\pi_AS(X))$ be
the natural transformation corresponding to $\eta$. It is clear that
$\eta^*|_{\im \pi_A}$ is an isomorphism. Since $\cR_A$ is the
triangulated hull of $\im\pi_A$, one deduces that $\eta^*$ is an
isomorphism over $\cR_A$. In particular, $D\cR_A(?,\pi_AS(X))$ is
representable. On the other hand, the Serre functor $\wt{S}$ implies
$D\cR_A(?,\wt{S}\pi_AX)$ is also represented by $\cR_A(\pi_AX,?)$.
Thus, we have $\pi_AS(X)\cong \wt{S}\pi_AX$.

Let $\Sigma^{-1}SX\xrightarrow{f} Y\xrightarrow{g}
X\xrightarrow{h}S(X)$ be an AR-triangle of $\der^b(\mod A)$. Let
$\pi_A(\Sigma^{-1}SX)\xrightarrow{u} W\to
\pi_AX\xrightarrow{v}\pi_AS(X)$ be the AR-triangle in $\cR_A$.
Clearly, $\pi_A(f)$ is not a split monomorphism. Thus, by the
definition of AR-triangle, there is a morphism $t:W\to \pi_AY$ such
that $\pi_A(f)=t\circ u$. Namely, we have the following commutative
diagram of triangles
\[\xymatrix{\pi_A(\Sigma^{-1}SX)\ar@{=}[d]\ar[r]^u&W\ar[d]^t\ar[r]&\pi_AX\ar[d]^s\ar[r]^{v}&\pi_AS(X)\ar@{=}[d]\\
\pi_A(\Sigma^{-1}SX)\ar[r]^{\pi_A(f)}&\pi_A(Y)\ar[r]^{\pi_A(g)}&\pi_AX\ar[r]^{\pi_A(h)}&\pi_AS(X)}
\]
We claim that $s$ is an isomorphism. Suppose not, then $s$ is
nilpotent by the indecomposable of $X$. Then $\pi_A(h)\circ s=0$
follows form that $\Sigma^{-1}SX\xrightarrow{f} Y\xrightarrow{g}
X\xrightarrow{h}S(X)$ is an AR-triangle, which implies $v=0$,
contradiction. Thus, $t$ is isomorphism. In particular, the image of
$\Sigma^{-1}SX\xrightarrow{f} Y\xrightarrow{g} X\xrightarrow{h}S(X)$
is indeed an AR-triangle of $\cR_A$.

 Now one can easily deduce that there is no irreducible morphism
between $\im \pi_A$ and $\cR_A\backslash \im\pi_A$, which completes
the proof.
\end{proof}
\begin{remark}
Theorem~\ref{t:separated} has been proved for the generalized
cluster category in ~\cite{AmiotOppermann2010} by using different
approach. We remark that one can adapt a variant proof to deduce the
result for generalized cluster category. Indeed, by the
$2$-Calabi-Yau property of generalized cluster category, one can
deduce that the Serre functor of the derived category coincides with
the Serre functor of the generalized cluster category on the
objects. Then one shows that the functor $\pi_A$ preserves
AR-triangles. By the universal property of root category, the Serre
functor $S:\der^b(\mod A)\to \der^b(\mod A)$ will induce a functor
$\ol{S}:\cR_A\to \cR_A$. It would be interesting to compare it with
the Serre functor $\wt{S}$.
\end{remark}
\subsection{The ADE root categories}
Let $A$ and $B$ be finite-dimensional $k$-algebras with finite
global dimension. If $A$ and $B$ are derived equivalent, it is clear
that $\cR_A\cong \cR_B$. It would be interesting to characterize all
the algebras which have the same root category up to triangle
equivanlence. In general, this question seems to be very hard. In
the following we will characterize the algebras share the root
category with a path algebra of Dynkin quiver. Since the derived
category of Dynkin quiver is not dependent on the choice of
orientation, we assume $Q$ be the following quiver for simplicity.
\[\xymatrix@R=0.3cm{A_n:&1\ar[r]&2\ar[r]&\cdots\ar[r]&n-1\ar[r]&n\\
&1\ar[rd]\\D_n:&&3\ar[r]&4\ar[r]&\cdots\ar[r]&n\\&2\ar[ru]\\
&&&3\ar[d]\\
E_6:& 1\ar[r]&2\ar[r]&4\ar[r]&5\ar[r]&6\\
&&&3\ar[d]\\
E_7:&1\ar[r]&2\ar[r]&4\ar[r]&5\ar[r]&6\ar[r]&7\\
&&&3\ar[d]\\
E_8:&1\ar[r]&2\ar[r]&4\ar[r]&5\ar[r]&6\ar[r]&7\ar[r]&8}
\]
\begin{theorem}\label{t:ADE root category}
Let $A$ be a finite-dimensional $k$-algebra with finite global
dimension. If  the root category $\cR_A\cong \cR_{kQ}$ for some
Dykin quiver $Q$, then $A$ is derived equivalent to $kQ$.
\end{theorem}
\begin{proof}
Since $Q$ is finite Dykin quiver, the AR-quiver of $\der^b(\mod kQ)$
is connected. The canonical functor $\pi_{kQ}:\der^b(\mod kQ)\to
\cR_{kQ}$ is dense, which implies that the AR-quiver of $\cR_{kQ}$
is connected. By theorem~\ref{t:separated}, we inform that the
functor $\pi_A:\der^b(\mod A)\to \cR_A$ is dense. In particular, any
$X\in \cR_A$ has preimage in $\der^b(\mod A)$. Let $P_i,
i=1,\cdots,n$ be the indecomposable projective $kQ$-modules. It is
clear that
\[\dim_k\cR_{kQ}(\pi_{kQ}P_i,\pi_{kQ}P_j)\leq 1 \ \text{for}\ i\leq j \ \text{and}\
\cR_{kQ}(\pi_{kQ}P_i,\pi_{kQ}P_j)=0\ \text{for}\ i>j.
\]
Let $F:\cR_{kQ}\to \cR_A$ be the triangle equivalent functor. We
claim that there is an object $M=M_1\oplus\cdots\oplus M_n$ in
$\der^b(\mod A)$ such that \[\pi_A(M)=F(\pi_{kQ}(kQ)) \text{and}\
\der^b(\mod A)(M,\Sigma^tM)=0\ \text{for}\ t\neq 0.\]

 Let
$\{\Sigma^{2r}X_i|r\in \Z\}$ be the preimages of $F(\pi_{kQ}(P_i))$
in $\der^b(\mod A)$. Let us prove this claim for case $Q=A_n$, the
other cases are similar. Note that $n$ is a sink vertex, we can
choose $M_n=X_n$. Since $\cR_A(F(\pi_{kQ}(P_{n-1})),
F(\pi_{kQ}(P_n)))\cong k$, there is a unique $r_{n-1}\in \Z$ such
that $\der^b(\mod A)(\Sigma^{2r_{n-1}}X_{n-1}, M_n)\cong k$ and
$\der^b(\mod A)(\Sigma^{2t}X_{n-1}, X_n)=0$ for $t\neq r_{n-1}$. We
can take $M_{n-1}=\Sigma^{2r_{n-1}}X_{n-1}$. Replace $M_n$ by
$M_{n-1}$, one can construct $M_{n-2}$ uniquely. For any nonzero
$f:M_{n-2}\to M_{n-1}, g:M_{n-1}\to M_n$,the composition $g\circ
f\neq 0$. Inductively, one can construct $M_i$ for any $i=1, \cdots,
n$. Clearly, we have $\pi_A(M)\cong F(\pi_{kQ}kQ)$ and $\der^b(\mod
A)(M, \Sigma^{2r}M)=0$ for $r\neq 0$. $\der^b(\mod
A)(M,\Sigma^{2r+1}M)=0, r\in Z$ follows from
$\cR_{kQ}(\pi_{kQ}kQ,\Sigma\pi_{kQ}kQ)=0$. In particular, $M$ is a
(partial) tilting complex of $\der^b(\mod A)$. We have $\der^b(\mod
kQ)\cong \der^b(\mod \End_{\der^b(\mod A)}(M))\cong \tria(M)$, where
$\tria (M)$ is the thick subcategory of $\der^b(\mod A)$ contains
$M$. If we can show that $\tria(M)=\der^b(\mod A)$, then we are
done. Let $i:\der^b(\mod kQ)\xrightarrow{\sim}
\tria(M)\hookrightarrow \der^b(\mod A)$ be the composition. By the
universal property of root category, we have the following
commutative diagram
\[\xymatrix{\der^b(\mod kQ)\ar[d]^{\pi_{kQ}}\ar@{^{(}->}[r]^i&\der^b(\mod A)\ar[d]^{\pi_A}\\
\cR_{kQ}\ar[r]^{\ol{i}}&\cR_A}
\]
where $\ol{i}$ is induced by the full embedding $i$. It is clear
that $\ol{i}$ is also full and faithful, thus an equivalence, which
implies $i$ is dense and an equivalence.
\end{proof}
\subsection{Tame quiver of type $\wt{D}$ and $\wt{E}$}
Assume $Q$ be the following quiver
\[\xymatrix@R=0.3cm{&&2\ar[d]&&&n-1\ar[d]\\
\wt{D_n}:&1\ar[r]&3\ar[r]&\cdots\ar[r]&n-2\ar[r]&n\ar[r]&n+1\\
&&&3\ar[d]\\&&&4\ar[d]\\
\wt{E_6}:&1\ar[r]&2\ar[r]&5\ar[r]&6\ar[r]&7\\
&&&&4\ar[d]\\
\wt{E_7}:&1\ar[r]&2\ar[r]&3\ar[r]&5\ar[r]&6\ar[r]&7\ar[r]&8\\
&&&3\ar[d]\\
\wt{E_8}:&1\ar[r]&2\ar[r]&4\ar[r]&5\ar[r]&6\ar[r]&7\ar[r]&8\ar[r]&9}
\]
The theorem ~\ref{t:ADE root category} also holds for tame quiver of
type $\wt{D}$ and $\wt{E}$. One can adapt a variant proof of theorem
~\ref{t:ADE root category}.
\begin{proposition}
Let $A$ be a finite dimensional $k$-algebra with finite global
dimension. If  the root category $\cR_A\cong \cR_{kQ}$ for some tame
quiver $Q$ of type $\wt{D}\wt{E}$, then $A$ is derived equivalent to
$kQ$.
\end{proposition}
\begin{proof}
It suffices to prove this proposition for $Q$ be the above quiver.
Clearly the canonical functor $\pi_{kQ}:\der^b(\mod kQ)\to \cR_{kQ}$
is dense. In this case, $\der^b(\mod kQ)$ is the union of
preprojective component, preinjective component and tubes up to
shift. If the image $\im \pi_A$ intersects with preprojective (resp.
preinjective) component nonempty, by theorem~\ref{t:separated},
every object in this component belongs to $\im \pi_A$. Then one can
adapt the proof of theorem ~\ref{t:ADE root category}. Now suppose
that $\im \pi_A$ intersects with both preprojective  and
preinjective component empty. Let $T$ be the union of $kQ$-modules
in the tubes. It is clear that $T$ is a  hereditary abelian
subcategory of $kQ$-modules. By theorem $9.1$ of ~\cite{Keller2005},
we know that $\der^b(T)/\Sigma^2$ is triangulated and we have the
following commutative diagram
\[\xymatrix{\der^b(T)\ar[d]^{\pi}\ar@{^{(}->}[r]^i&\der^b(\mod kQ)\ar[d]^{\pi_{kQ}}\\
\der^b(T)/\Sigma^2\ar[r]^{\ol{i}}&\cR_{kQ}}
\]
where $\ol{i}$ is induced by $i$. In particular, we know that
$\ol{i}$ is a full embedding. Now $\im\pi_A\subset T\cup \Sigma T$
implies that $\im\pi_A\subset \der^b(T)/\Sigma^2$, which contradicts
to $\tria(\im\pi_A)=\cR_A$.
\end{proof}

\section{Ringel-Hall Lie algebras and GIM Lie algebras}
Throughout this section, let $k$ be a field with $|k|=q$. We study
the Ringel-Hall Lie algebras of a class of finite-dimensional
$k$-algebras with global dimension $2$. Building on the
representation theory of these algebras, we will give a negative
 answer for a question on GIM-Lie algebras by Slodowy
in~\cite{Slodowy1986}. We remark that different counterexamples of
this question have been discovered by Alpen~\cite{Alpen1984} by
considering fixed point subalgebras of certain Lie algebras.
\subsection{Generalized intersection matrix Lie algebras}
We recall the generalized intersection matrix Lie algebra (GIM-Lie
algebra for short) following Slodowy~\cite{Slodowy1986}. A matrix
$A\in M_l(\Z)$ is called a {\it generalized intersection matrix}, or
GIM for short, if the followings are satisfied
\begin{eqnarray*}
&&A_{ii}=2\\
&&A_{ij}<0 \Longleftrightarrow A_{ji}<0\\
&&A_{ij}>0 \Longleftrightarrow A_{ji}>0
\end{eqnarray*}
If moreover $A$ is symmetric, then $A$ is called an {\it
intersection matrix}. Given a GIM $A\in M_l(\Z)$, a {\it root basis}
associated to $A$ is a triplet $(H,\triangledown,\vartriangle)$
consisting of
\begin{itemize}
\item[$\circ$] a finite dimensional $\Q$-vector space $H$;
\item[$\circ$] a family $\triangledown=\{\alpha_1^{\vee}, \cdots, \alpha_l^{\vee}\}$, where $\alpha_i^{\vee}\in
H$;
\item[$\circ$] a family $\vartriangle=\{\alpha_1, \cdots, \alpha_l\}$,
where $\alpha_i\in H^*=\Hom_{\Q}(H,\Q)$
\end{itemize}
satisfy the following
\begin{itemize}
\item[1)] both sets $\vartriangle$ and $\triangledown$ are linearly
independent;
\item[2)] $\alpha_j(\alpha_i^{\vee})=A_{ij}$ for all $1\leq i,j\leq l$;
\item[3)] $\dim_{\Q}H=2l-\rank A$.
\end{itemize}

 The {\it GIM-Lie algebra} $\mg=GIM(A)$ attached to the root basis $(H,
\triangledown, \vartriangle)$ is given by the generators
$\mh=H\otimes_{\Q}\C$ and $e_{\pm \alpha}, \alpha\in \vartriangle$
satisfying the following relations:
\begin{eqnarray*}
&(1)&[h,h']=0, h,h'\in \mh\\
&(2)&[h,e_{\alpha}]=\alpha(h)e_{\alpha}, h\in \mh, \alpha\in \pm \vartriangle\\
&(3)&[e_{\alpha}, e_{-\alpha}]=\alpha^{\vee}, \alpha\in \vartriangle\\
&(4)&ad(e_{\alpha})^{max(1,1-\beta(\alpha^{\vee}))}e_{\beta}=0,
\alpha\in \vartriangle, \beta\in \pm\vartriangle\\
&(5)&ad(e_{-\alpha})^{max(1,1-\beta(-\alpha^{\vee}))}e_{\beta}=0,
\alpha\in \vartriangle, \beta\in \pm\vartriangle.
\end{eqnarray*}
 If $A$ is a symmetrizable
generalized Cartan matrix, then the $GIM(A)$ is essentially the
Kac-Moody algebras associated to $(H, \triangledown,\vartriangle)$.

Let $ad: \mg\to \End(\mg)$ be the adjoint representation of $\mg$.
Consider the restriction of $ad$ to $\mh$, the Lie algebra $\mg$
decomposes into a direct sum
\[\mg=\bigoplus_{\gamma\in \mh^*}\mg_{\gamma}
\]
of eigenspaces
\[\mg=\{x\in \mg|[h,x]=\gamma(h)x \ \text{for all}\ h\in \mh\}.
\]
Clearly, we have $\mh\subseteq \mg_{0}$. The following question has
been addressed in ~\cite{Slodowy1986} by Slodowy: Does equality
hold?

If we consider the derived subalgebra $[\mg, \mg]$ of $\mg$, the
above question is equivalent to the following: Do we have $\dim_{\C}
[\mg, \mg]_{0}=l$?  We remark that the derived subalgebra $[\mg,
\mg]$ can be presented by generators $\alpha_i^{\vee}, 1\leq i\leq
l$ and $e_{\alpha}, \alpha\in \pm\vartriangle$ with the same
relations in $\mg$. In~\cite{Alpen1984}, Alpen  has given a negative
answer for this question by using Lie theory. In the following, we
will give a negative answer for this question via
representation-theoretic approach.

\subsection{The Ringel--Hall Lie algebra}
We recall the definition of the Ringel--Hall Lie algebra of a
$2$-periodic triangulated category following~\cite{Peng-Xiao2000}.
Let $\mathcal{R}$ be a Hom-finite $k$-linear triangulated category
with suspension functor $\Sigma$. By $\ind \mathcal{R}$ we denote a
set of representatives of the isoclasses of all indecomposable
objects in $\mathcal{R}$.

Given any objects $X,Y,L$ in $\mathcal{R}$, we define
\begin{eqnarray*}W(X,Y;L)&=&\{(f,g,h)\in \Hom_{\mathcal{R}}(X,L)\times \Hom_{\mathcal{R}}(L,Y)\times \Hom_{\mathcal{R}}(Y,\Sigma X)|\\
&&X\xrightarrow{f}L\xrightarrow{g} Y\xrightarrow{h}\Sigma X \text{
is a triangle}\}.
\end{eqnarray*}
The action of $\Aut(X)\times \Aut(Y)$ on $W(X,Y;L)$ induces the
orbit space
\[V(X,Y;L)=\{(f,g,h)^{\wedge}|(f,g,h)\in W(X,Y;L)\}
\]
where
\[(f,g,h)^{\wedge}=\{(af, gc^{-1}, ch(\Sigma a)^{-1})|(a,c)\in \Aut(X)\times\Aut(Y)\}.
\]
Let $\Hom_{\mathcal{R}}(X,L)_Y$ be the subset of
$\Hom_{\mathcal{R}}(X,L)$ consisting of
 morphisms $l:X\to L$ whose mapping cone $Cone(l)$ is isomorphic to $Y$. Consider the action of the
 group $\Aut(X)$ on $\Hom_{\mathcal{R}}(X,L)_Y$ by $d\cdot l=dl$, the orbit is denoted by $l^*$ and the
 orbit space is denoted by $\Hom_{\mathcal{R}}(X,L)_Y^*$.
  Dually one can also consider the subset $\Hom_{\mathcal{R}}(L,Y)_{\Sigma X}$
  of $\Hom_{\mathcal{R}}(L,Y)$ with the group action $\Aut(Y)$ and the orbit
   space $\Hom_{\mathcal{R}}(L,Y)_{\Sigma X}^*$. The following proposition is an observation due to
   ~\cite{XiaoXu2008}.
\begin{lemma}\label{l:hall-num}
$|V(X,Y;L)|=|\Hom_{\mathcal{R}}(X,L)_Y^*|=|\Hom_{\mathcal{R}}(L,Y)_{\Sigma
X}^*|$.
\end{lemma}

We assume further that $\mathcal{R}$ is \emph{$2$-periodic}, \ie
$\mathcal{R}$ is Krull--Schmidt and $\Sigma^2\cong 1$.

Let $\go(\mathcal{R})$ be the Grothendieck group of $\mathcal{R}$
and $I_{\mathcal{R}}(-,-)$ the symmetric Euler form of
$\mathcal{R}$. For an object $M$ of $\mathcal{R}$, we denote by
$[M]$ the isoclass of $M$ and by $h_M=\dimv M$ the canonical image
of $[M]$ in $\go(\mathcal{R})$. Let $\mh$ be the subgroup of
$\go(\mathcal{R})\otimes_{\Z}\Q$ generated by $\frac{h_M}{d(M)},
M\in \ind \mathcal{R}$, where $d(M)=\dim_k(\End(X)/rad \End(X))$.
One can naturally extend the symmetric Euler form to $\mh\times
\mh$. Let $\mn$ be the free abelian group with basis $\{u_X|X\in
\ind \mathcal{R}\}$. Let
\[\mg(\mathcal{R})=\mh\oplus\mn,
\]
a direct sum of $\Z$-modules. Consider the quotient group
\[\mg(\mathcal{R})_{(q-1)}=\mg(\mathcal{R})/(q-1)\mg(\mathcal{R}).
\]
Let $F_{YX}^L=|V(X,Y;L)|$. Then by Peng and Xiao
~\cite{Peng-Xiao2000} we know that $\mg(\mathcal{R})_{(q-1)}$ is a
Lie algebra over $\Z/(q-1)\Z$, called the \emph{Ringel--Hall Lie
algebra} of $\mathcal{R}$. The Lie operation is defined as follows.
\begin{itemize}
\item[(1)] for any indecomposable objects $X,Y\in \mathcal{R}$,
\[
[u_X,u_Y]=\sum_{L\in \ind
\mathcal{R}}(F_{YX}^L-F_{XY}^L)u_L-\delta_{X,\Sigma
Y}\frac{h_X}{d(X)},
\]
where $\delta_{X,\Sigma Y}=1$ for $X\cong \Sigma Y$ and $0$ else.
\item[(2)] $[\mh, \mh]=0.$
\item[(3)] for any objects $X,Y\in \mathcal{R}$ with $Y$ indecomposable,
\[[h_X,u_Y]=I_{\mathcal{R}}(h_X,h_Y)u_Y,\qquad [u_Y,
h_X]=-[h_X,u_Y].
\]
\end{itemize}
A  triangulated category $\ct$ is called {\it proper}, if for any nonzero indecomposable object $X\in \ct$, $\dimv X\neq 0 $ in the Grothendieck group $\go(\ct)$. If the $2$-periodic triangulated category $\cR$ is proper, then $[u_X, u_{\Sigma X}]=-\frac{h_X}{d(X)}$, which coincides the origin definition in~\cite{Peng-Xiao2000}. However, the proof in ~\cite{Peng-Xiao2000} is still valid for non-proper $2$-periodic triangulated category for the Lie bracket defined above ({\it cf.} \cite{Xiao-Xu-Zhang2006}).

\subsection{A class of finite-dimensional $k$-algebras}
Let $Q$ be the following quiver
\[\xymatrix{&&&& 0\ar@<1ex>[d]^{\alpha}\ar@<-1ex>[d]\\&n
&n-1\ar[l]&2\ar@{.>}[l]&1\ar[r]^\beta\ar[l]^{\gamma}&n+1\ar@{.>}[r]&\circ\ar[r]&n+m}
\]
We assume $m\geq 1, n\geq 2$.  Let $A$ be the quotient algebra of
path algebra $kQ$ by the ideal generated by $\beta\circ \alpha,
\gamma\circ \alpha$. It has global dimension $2$.

Let $E$ be a field extension of $k$ and set $V^E=V\otimes_kE$ for
any $k$-space $V$. Then $A^E$ is an $E$-algebra and, for $M\in \mod
A$, $M^E$ has a canonical $A^E$-module structure.  Clearly, $A^E$
still has global dimension $2$. Let $\cR_{A^E}$ be the root category
of $A^E$. Thus, one has the Ringel-Hall Lie algebra
$\mg(\cR_{A^E})_{(|E|-1)}$, which is a Lie algebra over
$\Z/(|E|-1)\Z$.

Let $\ol{k}$ be the algebraic closure of $k$ and set
\[\Omega=\{E|k\subseteq E\subseteq \ol{k}\ \text{ is a finite field
extension}\}.
\]
We consider the direct product $\prod_{E\in
\Omega}\mg(\cR_{A^E})_{(|E|-1)}$ of Lie algebras and let
$\cl\mg(\cR_A)$ be the Lie subalgebra of $\prod_{E\in
\Omega}\mg(\cR_{A^E})_{(|E|-1)}$ generated by
$u_{S_i}=(u_{S_i^E})_{E\in \Omega}$ and $u_{\Sigma S_i}=(u_{\Sigma
S_i^E})_{E\in \Omega}$ for all simple $A$-modules $S_i, 0\leq i\leq
m+n$. Clearly, $h_i=(h_{S_i^E})_{E\in \Omega}, 0\leq i\leq n+m$
belong to $\cl\mg(\cR_A)$. We call $\cl\mg(\cR_A)$ the {\it integral
Ringel-Hall Lie algebra} of $A$. Clearly, the algebra
$\cl\mg(\cR_A)$ has a grading by the Grothendieck group $\go(\cR_A)$
of $\cR_A$, namely,
\[\cl\mg(\cR_A)=\bigoplus_{\alpha\in
\go(\cR_A)}\cl\mg(\cR_A)_\alpha
\]
such that $\deg u_{S_i}=\dimv S_i, \deg (u_{\Sigma S_i})=\dimv
\Sigma S_i$. In particular, $h_i\in \cl\mg(\cR_A)_0$.

 Let $(-,-)$ be the symmetric Euler form of $\cR_A$ ({\it cf.}
section~\ref{s:grothendieck group}). Then image of simple
$A$-modules $[S_i], 0\leq i\leq n+m$ form a $\Z$-basis of
$\go(\cR_A)$. Define the matrix $C=(c_{ij})$, $c_{ij}=([S_i],[S_j])$
for $0\leq i, j\leq n+m$. One can easily show that $C$ is an
intersection matrix. Let $(H,\triangledown, \vartriangle)$ be a root
basis of $C$. Thus one can form the GIM Lie algebra $\mg(C)=GIM(C)$
associated to $C$. We are interested in its derived subalgebra
$\mg(C)'=[\mg(C), \mg(C)]$.
\begin{theorem}~\label{t:Lie algebra homom} There is a surjective Lie algebra homomorphism $\phi:\mg(C)'\to
\cl\mg(\cR_A)\otimes_{\Z}\C$ defined by
\begin{eqnarray*}
&&\alpha_i^{\vee}\mapsto h_i,\\
&&e_{\alpha_i}\mapsto u_{S_i}\\
&&e_{-\alpha_i}\mapsto -u_{\Sigma S_i}, 0\leq i\leq n+m.
\end{eqnarray*}
Moreover, $\phi$ keeps the gradations and $\dim_{\C}
(\cl\mg(\cR_A)\otimes_{\Z}\C)_0\geq m+n+2$. As a consequence, we infer
that $\dim_{\C}\mg(C)'_0\geq m+n+2$.
\end{theorem}
The proof of this theorem carries throughout the rest of this
section.
\begin{lemma}\label{l:dimension of 0th gradation}
Let $M$ be the unique indecomposable $A$-module with composition
series $S_0,S_1,S_2,S_{n+1}$. Then $u_{M}=(u_{M^E})_{E\in \Omega}\in
\cl\mg(\cR_A)$ and $0\neq [u_{M}, u_{\Sigma M}]\not\in
\wt{\mh}\otimes_{\Z}\C$, where $\wt{\mh}$ is the subspace of
$\cl\mg(\cR_A)$ spanned by $h_i, 0\leq i\leq n+m$.
\end{lemma}
\begin{proof}
One can easily check that
$u_{M^E}=[[[u_{S_0^E},u_{S_1^E}],u_{S_2^E}],u_{S_{n+1}^E}]$ by using
lemma~\ref{l:hall-num} for any $E\in \Omega$. Thus, both $u_M,
u_{\Sigma M}$ belong to $\cl\mg(\cR_A)$. Let $P_i$ be the
indecomposable projective $A$-modules associated to each vertex $i$.
Let $\to P_0\to P_1\to P_2\oplus P_{n+1}\to M\to 0$ be the
projective cover of $M$.  We infer that
$\Hom_{\cR_A}(M,M)=\Hom_{\cd^b(\mod A)}(M,M)\oplus \Hom_{\cd^b(\mod
A)}(M,\Sigma^2 M)$. Moreover, $\dim_k\Hom_{\cR_A}(M,M)=2$ and
$\dim_k\rad \Hom_{\cR_A}(M,M)=1$.

Now consider triangles in $\cR_A$
\[M\to L\to \Sigma M\xrightarrow{f} \Sigma M \ \text{and}\ \Sigma M\to N\to
M\xrightarrow{g}\Sigma^2 M,
\]
we can write $f=f_0+f_1$, where $f_0\in \Hom_{\cd^b(\mod A)}(\Sigma
M,\Sigma M)$ and $f_1\in \Hom_{\cd^b(\mod A)}(\Sigma M,\Sigma^3M)$.
If $f_0\neq 0$, then $f$ is an isomorphism and $L\cong 0$. Thus, it
suffices to consider for $f_0=0$, {\it i.e.} $0\neq f\in \rad
\Hom_{\cR_A}(M,M)$, and then the triangle $M\to L\to\Sigma
M\xrightarrow{f}\Sigma M$ is induced by a triangle in $\cd^b(\mod
A)$. By computing the mapping cone of $f$ in $\cd^b(\mod A)$, we
infer that $L$ isomorphic to the complex $\cdots \to 0\to
P_0\xrightarrow{(f,l)}M\oplus P_1\to P_2\oplus P_3\to 0\cdots$,
where $P_2\oplus P_3$ is in the $-1$-th component. We claim that $L$
is indecomposable in $\cd^b(\mod A)$. Indeed, suppose $L\cong
X\oplus Y$ in $\cd^b(\mod A)$. Then $H^{*}(L)\cong H^{*}(X)\oplus
H^*(Y)$, where $H^*(-)$ be the homology groups of corresponding
complex. Now the only nonzero homology groups of $L$ are
$H^{-1}(L)\cong H^{-2}(L)\cong M$, which are indecomposable
$A$-modules. Thus, we may assume $X\cong \Sigma^2 M$ and $Y\cong
\Sigma M$. Now in the root category $\cR_A$, we have $\Sigma^2M\cong
M$. In particular, we get a triangle $M\to M\oplus \Sigma M\to
\Sigma M\xrightarrow{f} \Sigma M$. By a well-known fact, the
triangle is split and $f=0$, contradiction.  Since $\dim_{k}\rad
\Hom_{\cR_A}(M,M)=1$, for any $f, h\in \rad \Hom_{\cR_A}(M,M)$, the
mapping of $f$ and $h$ are isomorphic to each other.

Similarly, one can discuss for $g$ and show that $N$ is
indecomposable if and only if $0\neq g\in \Hom_{\cd^b(\mod
A)}(M,\Sigma^2 M)$. In this case, we have $N\cong \Sigma^{-1}L$. Now
by the definition of the Lie bracket, we have
\begin{eqnarray*}[u_M,u_{\Sigma M}]&=&-h_{M}+\sum_{L\in \ind \cR_A}(F_{\Sigma M,
M}^L-F_{M,\Sigma M}^L)u_L\\
&=&-h_M+F_{\Sigma M, M}^Lu_L-F_{M,\Sigma
M}^{\Sigma^{-1}L}u_{\Sigma^{-1} L}
\end{eqnarray*}
One can show that $\dim_k\Hom_{\cR_A}(M,L)=1$. Therefore, by
lemma~\ref{l:hall-num} we have $F_{\Sigma M, M}^L=F_{M,\Sigma
M}^{\Sigma^{-1}L}=1$. In particular, we have $[u_M, u_{\Sigma
M}]=-h_M+u_L-u_{\Sigma L}$ in $\mg(\cR_A)_{(q-1)}$. We remark that
the proof above is valid for any finite field extension of $k$. Thus,
in the integral Ringel-Hall Lie algebra $\cl\mg(\cR_A)$, we also
have $[u_M, u_{\Sigma M}]=-h_M+u_{L}-u_{\Sigma L}$, which implies
the desired result.
\end{proof}

Now we are in a position to prove the theorem~\ref{t:Lie algebra
homom}.
\begin{proof}  The relations
$(1)(2)(3)$ follows from the definition of Lie bracket of
Ringel-Hall Lie algebra. It suffices to show $u_{S_i}, u_{\Sigma
S_j}, 0\leq i,j\leq m+n$ satisfy the relations $(4)$ and $(5)$. We
discuss for $i,j$ in $4$ cases.
\begin{itemize}
\item[Case 1:] $i,j\in \{0, 1\}$. We consider the quotient algebra
$B=A/A(e_{2}+e_3+\cdots +e_{n+m})A$, where $e_i$ is the idempotent
associated to the vertex $i$. Note that $B$ is projective as right
$A$-module. Then the derived functor $F=-\lten_BB_A: \cd^b(\mod
B)\to \cd^b(\mod A)$ is an embedding by theorem ~3.1 in
~\cite{Cline-Parshall-Scott}. Now, lemma~\ref{l:fully faithful}
implies the induced functor $\ol{F}:\cR_B\to \cR_A$ is also fully
faithful. In particular, we have a injective Lie algebra
homomorphism $\cl\mg(\cR_B)\to \cl\mg(\cR_A)$. Moreover, we can
identify the simple $B$-modules with simple $A$-modules via the
functor $F$. Thus, to check the a relations  for $\cl\mg(\cR_A)$
involve $0\leq i,j\leq 1$, it suffices to check it in
$\cl\mg(\cR_B)$. Note that the algebra $B$ is hereditary of type $\wt{A_1}$, we
infer that  $\cl\mg(\cR_B)\otimes_{\Z}\C$ is isomorphic to the affine Kac-Moody algebra of type $\wt{A_1}$ by the main theorem of ~\cite{Peng-Xiao2000}, which implies relations $(4)$ and $(5)$ hold.
\item[Case 2:] $i,j\in \{1, 2, \cdots, n+m\}$. Let $B=A/Ae_0A$. It
is easy to see that $\Ext_A(B_A,B_A)=0$. Again by
theorem~3.1~\cite{Cline-Parshall-Scott}, we have $F=-\lten_BB_A:
\cd^b(\mod B)\to \cd^b(\mod A)$ is an embedding. Note that in this
case $B$ is hereditary of Dynkin type $A_{m+n}$. Thus, the
Ringel-Hall algebra $\cl\mg(\cR_B)\otimes_{\Z}\C$ is isomorphic to
simple Lie algebra of type $A_{m+n}$. Now the result follows from
the proof of case $1$.
\item[Case 3:] $i=0, j\neq 1,2,n+1$. In particular, by the definition of Lie bracket we only need to show that $[u_{S_0},u_{S_j}]=0$ and $[u_{S_0}, u_{\Sigma S_j}]=0$. This follows from the fact that  $S_j$ has projective dimension $2$ and the projective resolution of $S_j$ does not involve $P_0$.
\item[Case 4:] $i, j\in \{0,2, n+1\}$. For the case $i=0, j=2$,  we consider the quotient algebra $B=A/A(e_3+\cdots+e_{m+n})A$, which turns out to be a tilted algebra of tame hereditary algebra of type $\wt{A_2}$. Thus the integral Ringel-Hall algebra $\cl\mg(\cR_B)\otimes_{\Z}\C$ is isomorphic to the Kac-Moody algebra of type $\wt{A_2}$. Now the result follows from the proof of case $1$, since we still have full embeddings $F:\der^b(\mod B)\to \der^b(\mod A)$ and $\ol{F}:\cR_B\to \cR_A$. For the case $i=0,j=n+1$, one considers the quotient algebra $B=A/A(e_2+\cdots +e_n+e_{n+2}+\cdots +e_{n+m})A$.
\end{itemize}
Thus $\phi$ is indeed a Lie algebra homomorphism. It is obviously
surjective and keeps the gradation. Clearly, $h_i=(h_{S_i^E})_{E\in
\Omega}$ is linearly independent in
$(\cl\mg(\cR_A)\otimes_{\Z}\C)_0$. By lemma~\ref{l:dimension of 0th
gradation}, we infer that $u_{M}-u_{\Sigma M}\in
(\cl\mg(\cR_A)\otimes_{\Z}\C)_0$, which is linearly independent to $h_0,h_1, \cdots, h_{m+n}$.
Thus, $\dim_{\C} (\cl\mg(\cR_A)\otimes_{\Z}\C)_0\geq m+n+2$.
This completes the proof.
\end{proof}

\begin{remark}
Firstly,
theorem~\ref{t:Lie algebra homom} essentially  give a negative answer to Slodowy's question. If the equality holds for $\mg=GIM(C)$, {\it i.e.} $\dim_{\C}\mg_0=m+n+2$, then for the derived subalgebra $\mg'=[\mg,\mg]$, we must have $\dim_{\C}\mg'_0=m+n+1$. In fact, following the proof of lemma~\ref{l:dimension of 0th gradation}, one can even show that $\dim_{\C} (\cl\mg(\cR_A)\otimes_{\Z}\C)_0\geq (m+1)n+1$.  Secondly, by lemma~\ref{l:dimension of 0th gradation}, we know that $(\dimv M,\dimv M)=4$ and $u_M\in \cl\mg(\cR_A)$. In particular, this also shows that the GIM Lie algebra $\mg$ has root with length greater than $2$. Thirdly, one can easily see that the root basis of $GIM(C)$ is braid equivalent to a root basis of affine Kac-Moody algebra of type $\wt{A}_{m+n}$. By theorem~\ref{t:Lie algebra homom}, we know that $GIM(C)$ is never isomorphic to an affine Kac-Moody algebra of type $\wt{A}_{m+n}$, this also show that the GIM Lie algebras are not invariant under braid equivalent in general.
\end{remark}

\begin{appendix}
\section{Recollement lives in root categories}
In the appendix, we show that a recollement of bounded derived
categories lives in the corresponding root categories under suitable
assumption. This can be use to construct various algebras
inductively such that the $2$-periodic orbit category is not
triangulated with the inherited triangle structure from the
bounded derived category.
\subsection{Derived category of dg category}
Let $\ca$ be a small differential graded (dg) $k$-category. We
identify a dg $k$-algebra with a dg category with one object. Let
$\Dif \ca$ be the dg category of right dg $\ca$-modules. A dg
$\ca$-module $P$ is called {\it $\ck$-projective} if $\Dif \ca(P,?)$
preserves acyclicity. For any dg category $\cb$, let
$\mathcal{Z}^0(\cb)$ be the category with the same objects of $\ca$
whose $\Hom$-space is given by
\[\mathcal{Z}^0(\cb)(X,Y)=Z^0(\cb(X,Y)),
\]
{\it i.e.} the 0th cocycle of dg $k$-module $\cb(X,Y)$. Let
$\ch^0(\cb)$ be the category with the same objects of $\cb$ whose
$\Hom$-space is given by
\[\ch^0(\cb)(X,Y)=H^0(\cb(X,Y)),
\]
{\it i.e.} the 0th homology of dg $k$-module $\cb(X,Y)$. For the dg
category $\Dif \ca$, we define $\cc(\ca):=\mathcal{Z}^0(\Dif \ca)$
and $\ch(\ca):=\ch^0(\Dif \ca)$. A morphism $L\to N$ in $\cc(\ca)$
is called quasi-isomorphism if it induces an isomorphism in
homology. Let $\der(\ca)$ be the derived category of $\ca$, {\it
i.e.} the localization of $\cc(\ca)$ with respect to the class of
quasi-isomorphism. A dg $\ca$-module $L$ is called {\it compact} if
$\der(\ca)(L,?)$ commutes with arbitrary direct sums. For instance,
the projective $\ca$-module $\ca(?,A), A\in \ca$ is both
$\ck$-projective and compact. Let $\per (\ca)$ be the perfect
derived category of $\ca$, {\it i.e.} the smallest subcategory of
$\der\ca$ contains $\ca$ and stable under shift,extensions and
passage to direct factors. For any subcategory $\cm\subseteq
\der(\ca)$, let $\tria(\cm)$ be the thick subcategory of $\der(\ca)$
contains $\cm$.

Let $X$ be a dg $\cb\otimes_k\ca\op$-module. It gives rise to a pair
of adjoint dg functors
\[\xymatrix{\Dif \ca\ar@<1ex>[r]^{T_X}&\Dif \cb.\ar@<1ex>[l]^{H_X}}
\]
Assume $X$ is $\ck$-projective  as $\cb\otimes_k\ca\op$-module, then
$(T_X,H_X)$ induces an adjoint pair triangle functors
$(\lt_X,\rh_X)$ over the derived categories, where $\lt_X$ is the
left derived functor of $T_X$. If both $\ca$ and $\cb$ are dg
$k$-algebras, we also write $?\lten_{\ca}X_{\cb}$ for $\lt_X$.

\subsection{The universal property of root category}
Let $A$ and $B$ be  finite-dimensional $k$-algebras with finite
global dimension. Let $F: \cd^b(\mod A)\to \cd^b(\mod B)$ be a
standard functor, {\it i.e.} $F\cong ?\lten_AX_B$ for some complex
of $A^{op}\otimes_k B$-module. Since for any triangle functor
$L:\der^b(\mod A)\to \der^b(\mod B)$, we have $L\circ
\Sigma^2_A\cong \Sigma^2_B\circ L$. By the universal property of dg
orbit category ({\it cf.} section $9.4$ in ~\cite{Keller2005}), $F$
naturally induces a triangle functor $\ol{F}:\cR_A\to \cR_B$ and we
have the following commutative diagram
\[\xymatrix{\der^b(\mod A)\ar[d]^{\pi_A}\ar[r]^F&\der^b(\mod B)\ar[d]^{\pi_B}\\
\cR_A\ar[r]^{\ol{F}}&\cR_B}
\]
where $\pi_A, \pi_B$ are the canonical functors. In the following,
we will study the induced functor $\ol{F}$ explicitly.

 We may
assume $_AX_B$ is $\ck$-projective as $A^{op}\otimes_kB$-module.
Clearly, $X$ has finite total homology. Moreover, $_AX_B$ is compact
as left $A$-module and right $B$-module respectively due to the fact
$A$ and $B$ have finite global dimension. Then we have the canonical
isomorphism $\RHom_{B}(_AX_B,?)\cong ?\lten_B \RHom_B(_AX_B,B)_A$.
Let $_BY_A\to _B\RHom_B(_AX_B,B)_A$ be a $\ck$-projective resolution
of $_B\RHom_B(_AX_B,B)_A$ as $B^{op}\otimes_k A$ -module. Thus, the
right adjoint $G$ of $F$ naturally isomorphic to $?\lten_B Y_A$.

 Let $\ca$ and $\cb$ be the dg category of bounded complexes of finitely
 generated projective $A$-modules and $B$-modules respectively. The tensor product by $X$ and $Y$ define
  dg functors $?\lten_A X:\ca\to \cb$ and $?\lten_BY:\cb\to
  \ca$. By abuse of notation, we denote these dg functors by $F$ and $G$ as
  well. Similarly, one can lift the square of the shift functors $\Sigma_A^2:\der^b(\mod A)\to \der^b(\mod
  A)$ and $\Sigma_B^2:\der^b(\mod B)\to \der^b(\mod B)$ to  dg
  functors $\Sigma_A^2:\ca\to \ca$ and $\Sigma_B^2:\cb\to\cb$.

  Let
  $\cR_{\ca}$ be the dg orbit category ({\it cf.} section $5$ of \cite{Keller2005}) of $\ca$ respects to $\Sigma_A^2$.
  Let $\cR_{\cb}$ be the dg orbit category of $\cb$ respects to
  $\Sigma_B^2$. We have canonical dg functors $\pi_{\ca}: \ca\to
  \cR_{\ca}$ and $\pi_{\cb}:\cb\to \cR_{\cb}$. We have natural
  isomorphisms $\Sigma_B^2\circ F\cong F\circ \Sigma_A^2$ and $\Sigma_A^2\circ G\cong G\circ \Sigma_B^2$ of dg
  functors. Thus, by the universal property of dg orbit categories,
  $F$ and $G$ induce  dg functors $\ol{F}:\cR_{\ca}\to \cR_{\cb}$
  and $\ol{G}:\cR_{\cb}\to \cR_{\ca}$. Clearly, $\ol{F}$ yields a
  $\cR_{\cb}\otimes_k\cR_{\ca}^{op}$-bimodule $X_{\ol{F}}$
  \[X_{\ol{F}}(B,A)\mapsto \cR_{\cb}(B,\ol{F}(A)).
  \]
  Similarly, $\ol{G}$ induces an
  $\cR_{\ca}\otimes_k\cR_{\cb}^{op}$-bimodule $Y_{\ol{G}}$
  \[Y_{\ol{G}}(A,B)\mapsto \cR_{\ca}(A, \ol{G}(B)).
  \]
Let $\lt_{X_{\ol{F}}}:\der(\cR_{\ca})\to \der(\cR_{\cb})$ be the
derived tensor functor of $X_{\ol{F}}$. Let
$\lt_{Y_{\ol{G}}}:\der(\cR_{\cb})\to \der(\cR_{\ca})$ be the derived
tensor functor of $Y_{\ol{G}}$. In the following, we identify the
objects of $\ca$ with $\cR_{\ca}$ and the objects of $\cb$ with
$\cR_{\cb}$ respectively.

  \begin{lemma}\label{l:adjoint}
  $\lt_{X_{\ol{F}}}$ is left adjoint to $\lt_{Y_{\ol{G}}}$.
  \end{lemma}
  \begin{proof}
  Clearly, $X_{\ol{F}}^A$ is $\ck$-projective for any $A\in \ca$ and $\lt_{X_{\ol{F}}}$ is left adjoint to $\rh_{X_{\ol{F}}}$. It suffices to
  show that $\lt_{Y_{\ol{G}}}\cong \rh_{X_{\ol{F}}}$. For any $\wt{A}\in
  \ca$, $X_{\ol{F}}(?,\wt{A})\cong \cR_{\cb}(?,\ol{F}(\wt{A}))$ which is
  compact in $\der(\cR_{\cb})$. By Lemma $6.2$ (a) in ~\cite{Keller1994}, we have
  $\lt_{X_{\ol{F}}^T}\cong \rh_{X_{\ol{F}}}$, where $X_{\ol{F}}^T$ is
  defined by
  \[X_{\ol{F}}^T(\wt{A},\wt{B})=\Dif \cR_{\cb}(X_{\ol{F}}(?,\wt{A}), \wt{B}^{\wedge})
  \]
  Thus, it suffices to show that we have a quasi-isomorphism $Y_{\ol{G}}\to
  X_{\ol{F}}^T$ as $\cR_{\ca}\otimes_k\cR_{\cb}^{op}$-bimodule.
 For any $\wt{A}\in \ca$ and $\wt{B}\in \cb$, we have
  \begin{eqnarray*}
  X_{\ol{F}}^T(\wt{A}, \wt{B})&=&\Dif \cR_{\cb}(X_{\ol{F}}(?,{\wt{A}}),\wt{B}^{\wedge})\\
  &=&
  \Dif \cR_{\cb}({\ol{F}}(\wt{A})^{\wedge}, \wt{B}^{\wedge})\\
  &\cong&
  \cR_{\cb}(\ol{F}(\wt{A}), \wt{B})\\
  &\cong&\bigoplus_{n\in \Z}\cb(F(\wt{A}), \Sigma_B^{2n}\wt{B})\\
  &\cong&\bigoplus_{n\in \Z} \RHom_{B}(\wt{A}\otimes_AX_B,
  \Sigma_B^{2n}\wt{B})\\
  &\cong&\bigoplus_{n\in \Z} \RHom_A(\wt{A},
  \RHom_B(X,\Sigma_B^{2n}\wt{B}))
  \end{eqnarray*}
Recall that we have quasi-isomorphism $\Sigma_B^{2n}\wt{B}\lten_B
  \RHom_B(X,B)\to \RHom_B(_AX_B,\wt{B})$ and $\wt{A}$ is
  $\ck$-projective as right $A$-module. In particular, we have a
  quasi-ismorphism
  \[\bigoplus_{n\in \Z}\RHom_A(\wt{A},\Sigma_B^{2n}\wt{B}\lten_B
  \RHom_B(X,B))\xrightarrow{q.is} \bigoplus_{n\in
  \Z}\RHom_A(\wt{A},\RHom_B(_AX_B,\wt{B})).
  \]
Again, we also have quasi-isomorphism
$\Sigma_B^{2n}\wt{B}\otimes_BY\to
\Sigma_B^{2n}\wt{B}\lten_B\RHom_B(X,B)$, which implies
\[\bigoplus_{n\in \Z} \RHom_{A}(\wt{A}, \Sigma_B^{2n}\wt{B}\otimes_B
  Y)\xrightarrow{q.is}\bigoplus_{n\in \Z}\RHom_A(\wt{A},\Sigma_B^{2n}\wt{B}\lten_B
  \RHom_B(X,B)).
\]
The first term
\begin{eqnarray*}
\bigoplus_{n\in \Z} \RHom_{A}(\wt{A}, \Sigma_B^{2n}\wt{B}\otimes_B
  Y)&\cong& \bigoplus_{n\in \Z} \RHom_A(\wt{A},
  \Sigma_A^{2n}(\wt{B}\otimes_BY))\\
  &\cong& \bigoplus_{n\in \Z}\ca(\wt{A},\Sigma_A^{2n}G(\wt{B}))\\
  &\cong &\cR_{\ca}(\wt{A}, \ol{G}(\wt{B}))\\
  &=&Y_{\ol{G}}(\wt{A}, \wt{B})
\end{eqnarray*}
Thus, we have obtained a quasi-isomorphism $Y_{\ol{G}}(\wt{A},
\wt{B})\to X_{\ol{F}}^T(\wt{A}, \wt{B})$, which is natural in both
$\wt{A}$ and $\wt{B}$. This completes the proof.
\end{proof}
The following lemma is quite obviously.
\begin{lemma}\label{l:fully faithful}
If $F:\der^b(\mod A)\to \der^b(\mod B)$ is fully faithful, then
$\lt_{X_{\ol{F}}}: \der(\cR_{\ca})\to \der(\cR_{\cb})$ is fully
faithful.
\end{lemma}
\begin{proof}
It follows from the Lemma $4.2$ (a) and (b) of ~\cite{Keller1994}
directly.
\end{proof}

Let $\cR_A$ be the perfect derived category of $\cR_{\ca}$. Let
$\cR_B$ be the perfect derived category of $\cR_{\cb}$. In other
word, $\cR_A$  and $\cR_B$ are the root categories of $A$ and $B$
respectively. Clearly, the triangle functors $\lt_{X_{\ol{F}}}$ and
$\lt_{Y_{\ol{G}}}$ restrict to an adjoint pair of triangle functors
\[\xymatrix{\cR_A\ar@<1ex>[r]^{\lt_{X_{\ol{F}}}}&\cR_B.\ar@<1ex>[l]^{\lt_{Y_{\ol{G}}}}}
\]
For simplicity, we still denote $\ol{F}:=\lt_{X_{\ol{F}}}:\cR_A\to
\cR_B$ and $\ol{G}:=\lt_{Y_{\ol{G}}}:\cR_B\to \cR_A$.

\subsection{Recollement lives in root categories} Suppose we are
given triangulated categories $\der', \der, \der''$ with triangle
functors
\[\xymatrix@C=1.5cm{\der'\ar[r]^-{i_*=i_!} &
\der\ar@<3ex>[l]^{i^!} \ar@<-3ex>[l]_-{i^* } \ar[r]^-{j^* =j^!}&
\der''. \ar@<3ex>[l]^-{j_*} \ar@<-3ex>[l]_-{j_!}}
\]
such that
\begin{itemize}
\item[$\circ$] $(i^*,i_*,i^!)$ and $(j_!,j^*,j_*)$ are adjoint
triples;
\item[$\circ$] $i_*, j_!,j_*$ are fully faithful;
\item[$\circ$] $j^*\circ i_*=0$;
\item[$\circ$] any $X$ in $\der$, there are distinguished
triangles
\[i_!i^!X\to X\to j_*j^*X\to \Sigma i_!i^!X\to X, j_!j^!X\to X\to
i_*i^*X\to \Sigma j_!j^!X
\]
where the morphisms $i_!i^!X\to , X\to j_*j^*X$, etc. are adjunction
morphisms.
\end{itemize}
Then we say that $\der$ admits {\it recollement} relative to $\der'$
and $\der''$. This notation was first introduced by
Beilinson-Bersstein-Deligne~\cite{Beilinson-Bersstein-Deligne1982}
in geometric setting with the idea that $ \der$ can be viewed as
bing glued together from $\der'$ and $\der''$. It is not hard to
show that if both $\der'$ and $\der''$ are Krull-Schmidt categories,
 so is $\der$. Recollement in algebraic setting was studied
 extensively due to the close relation with tilting theory~\cite{Jorgensen}\cite{Koenig1991}, etc.

Let $A,B,C$ are finite-dimensional $k$-algebras with finite global
dimension. Suppose that the bounded derived category $\der^b(\mod
B)$ admits a recollement relative to $\der^b(\mod A)$ and
$\der^b(\mod C)$. In particular, we have the following diagram of
triangulated categories and triangle functors
\[\xymatrix@C=1.2cm{\cd^b(\mod A)\ar[r]^-{i_*=i_!} &
\cd^b(\mod B)\ar@<3ex>[l]^-{i^!} \ar@<-3ex>[l]_-{i^* } \ar[r]^-{j^*
=j^!}& \cd^b(\mod C). \ar@<3ex>[l]^-{j_*} \ar@<-3ex>[l]_-{j_!}}
\]
Assume further that both the functors $i^*$ and $j_!$ are standard.
Then we have the following
\begin{theorem}\label{t:recollement}
Keep the notations above. Let $A$, $B$ and $C$ be finite-dimensional
$k$-algebras with finite global dimension such that the derived
category $\der^b(\mod B)$ admits a recollement relative to
$\der^b(\mod A)$ and $\der^b(\mod C)$. Assume that the functor $i^*$
and $j_!$ are standard. The root category $\cR_B$ admits a
recollement relative to $\cR_A$ and $\cR_C$. Moreover, we have the
following commutative diagram of recollements
\[\xymatrix@R=1.2cm@C=1.2cm{\cd^b(\mod A)\ar[d]^{\pi_A}\ar[r]^{i_*=i_!} &
\cd^b(\mod B)\ar[d]^{\pi_B}\ar@<3ex>[l]^-{i^!} \ar@<-3ex>[l]_-{i^* }
\ar[r]^-{j^* =j^!}& \cd^b(\mod C). \ar[d]^{\pi_C}\ar@<3ex>[l]^-{j_*}
\ar@<-3ex>[l]_-{j_!}\\
 \cR_A\ar[r]^-{\ol{i_*}=\ol{i_!}} & \cR_B\ar@<3ex>[l]^-{\ol{i^!}}
\ar@<-3ex>[l]_-{\ol{i^*} } \ar[r]^-{\ol{j^*} =\ol{j^!}}& \cR_C
\ar@<3ex>[l]^-{\ol{j_*}} \ar@<-3ex>[l]_-{\ol{j_!}}.}
\]
\end{theorem}
\begin{proof}
Since $i^*$ and $j_!$ are standard, then all the functors
$i_*,i^!,j^*,j_*$ are standard due to fact $A,B,C$ have finite
global dimension. Thus, we have the corresponding induced functors
$\ol{i^*}, \ol{i_*},\ol{i^!},\ol{j_!},\ol{j^*},\ol{j_*}$. The
commutativity of the above diagram follows from the universal
property of the root categories. It suffices to show that $\cR_B$
admits a recollement relative to $\cR_A$ and $\cR_C$ together with
the functors $\ol{i^*},
\ol{i_*},\ol{i^!},\ol{j_!},\ol{j^*},\ol{j_*}$. Clearly, $(\ol{i^*},
\ol{i_*},\ol{i^!})$ and $(\ol{j_!},\ol{j^*},\ol{j_*})$ are adjoint
triples follows Lemma ~\ref{l:adjoint}. By Lemma ~\ref{l:fully
faithful}, one infers that $\ol{i_*}, \ol{j_!},\ol{j_*}$ are fully
faithful. Since $\cR_A$ is generated by $\pi_A(A)$, to show that
$\ol{j^*}\circ \ol{i_*}=0$, it suffices to show $\ol{j^*}\circ
\ol{i_*}(\pi_A(A))=0$. By the commutativity of the above diagram,
this result follows from $j^*\circ i_*=0$. It remains to show that
for any $X\in \cR_B$ there are  triangles
\[\ol{i_!}\ol{i^!}X\to X\to \ol{j_*}\ol{j^*}X\to \Sigma
\ol{i_!}\ol{i^!}X,\  \ol{j_!}\ol{j^!}X\to X\to \ol{i_*}\ol{i^*}X\to
\Sigma \ol{j_!}\ol{j^!}X.
\]
We prove the existence of the first triangle, the second one is
similar.

If $X\in \im \pi_B$, there is $Y\in \der^b(\mod B)$ such that
$X=\pi_B(Y)$. By the recollement of $\der^b(\mod B)$ relative to
$\der^b(\mod A)$ and $\der^b(\mod C)$, we have
\[i_!i^!Y\to Y\to j_*j^*Y\to \Sigma i_!i^!Y.
\]
Applying the triangle functor $\pi_B$, we get a triangle in $\cR_B$
\[\pi_B(i_!i^!Y)\to \pi_B(Y)\to \pi_B(j_*j^*Y)\to \Sigma
\pi_B(i_!i^!Y).
\]
By the commutativity of the functors, we have
\[\ol{i_!}\ol{i^!}\pi_B(Y)\to \pi_B(Y)\to \ol{j_*}\ol{j^*}\pi_B(Y)\to \Sigma
\ol{i_!}\ol{i^!}\pi_B(Y).
\]
Clearly, this triangle is isomorphic to
\[\ol{i_!}\ol{i^!}\pi_B(Y)\xrightarrow{\eta_X} \pi_B(Y)\xrightarrow{\epsilon_X}\ol{j_*}\ol{j^*}\pi_B(Y)\to \Sigma
\ol{i_!}\ol{i^!}\pi_B(Y)
\]
where $\eta_X, \epsilon_X$ are adjunction morphisms, which implies
the later one is a distinguished triangle.

Consider the triangle $X\xrightarrow{f} Y \to Z\to \Sigma X$, where
$X, Y\in \im \pi_B$. Consider the following commutative square
\[\xymatrix{\ol{i_!}\ol{i^!}X\ar[d]^{\ol{j_*}\ol{j^*}f}\ar[r]^{\eta_X}&X\ar[d]^f\\
\ol{j_*}\ol{j^*}Y\ar[r]^{\eta_Y}&Y}
\]
By nine lemma, one can embed the square to the following commutative
diagram of triangles
\[\xymatrix{\ol{i_!}\ol{i^!}X\ar[d]^{\ol{j_*}\ol{j^*}f}\ar[r]^{\eta_X}&X\ar[d]^f\ar[r]^{\epsilon_X}&\ol{j_*}\ol{j^*}X\ar[d]^{\ol{j_*}\ol{j^*}f}\\
\ol{i_!}\ol{i^!}Y\ar[d]^{\ol{i_!}g_1}\ar[r]^{\eta_Y}&Y\ar[d]^g\ar[r]^{\epsilon_Y}&\ol{j_*}\ol{j^*}Y\ar[d]^{\ol{j_*}g_2}\\
\ol{i_!}U_Z\ar[r]^u&Z\ar[r]^v&\ol{j_*}V_Z}
\]
Let $\phi(u):U_Z\to \ol{i^!}Z $ be the morphism corresponds to $u$
under the natural isomorphism. Let $\phi(v):\ol{j^*}Z\to V_Z$ be the
morphism corresponds to $v$. It is clear that $\phi(u)$ and
$\phi(v)$ are isomorphisms. Thus, one gets the following commutative
diagram
\[\xymatrix{\ol{i_!}U_Z\ar[d]^{\ol{i_!}\phi(u)}\ar[r]^u&Z\ar@{=}[d]\ar[r]^v&\ol{j_*}V_Z\ar[d]^{\ol{j_*}(\phi(v)^{-1})}\ar[r]^w&\Sigma \ol{i_!}U_Z\ar[d]^{\Sigma \ol{i_!}\phi(u)}\\
 \ol{i_!}\ol{i^!}Z\ar[r]^{\eta_Z}&Z\ar[r]^{\epsilon_Z}&\ol{j_*}\ol{j^*}Z\ar[r]^\delta &\Sigma \ol{i_!}\ol{i^!}Z }
\]
where $\delta=\ol{j_*}\phi(v)\circ w\circ \Sigma \ol{i_!}\phi(u)$.
Thus, one informs that
\[\ol{i_!}\ol{i^!}Z\xrightarrow{\eta_Z} Z\xrightarrow{\epsilon_Z} \ol{j_*}\ol{j^*}Z\to \Sigma
\ol{i_!}\ol{i^!}Z
\]
is a distinguished triangle. Now this holds for any $Z\in \cR_B$ by
'devissage'.
\end{proof}

\begin{corollary}\label{c:dense}
Keep the assumptions in theorem~\ref{t:recollement} . If the canonical functor $\pi_B$ is
dense, then both $\pi_A$ and $\pi_C$ are dense.
\end{corollary}
\begin{proof}
For any $X\in \cR_A$, consider $\ol{i_*}X\in \cR_B$. By the dense of
$\pi_B$, there is a $Y\in \der^b(\mod B)$ such that $\pi_B(Y)\cong
\ol{i_*}X$. For $Y$, one have the canonical triangle $i_!i^!Y\to
Y\to j_*j^*Y\to$.  Applying the functor $\pi_B$, we have
\[\pi_B(i_!i^!Y)\to \pi_B(Y)\to \pi_B(j_*j^*Y)\to
\]
which have to isomorphic to the canonical triangle
\[\ol{i_!}\ol{i^!}(\ol{i_!}X)\to X\to 0\to.
\]
One gets $X\cong \pi_A(i^! Y)$. In particular, $\pi_A$ is dense.
Similar proof implies that $\pi_C$ is also dense.
\end{proof}
\begin{remark}
If only one of $i^*$ and $j_!$ is standard, say $i^*$ is standard.
Then lemma ~\ref{l:adjoint}  ~\ref{l:fully faithful} and a result of
~\cite{Parshall-Scott1988} imply that there is a  recollement
\[\xymatrix@C=1.5cm{
 \cR_A\ar[r]^-{\ol{i_*}=\ol{i_!}} & \cR_B\ar@<3ex>[l]^-{\ol{i^!}}
\ar@<-3ex>[l]_-{\ol{i^*} } \ar[r]^-{\ol{j^*} =\ol{j^!}}&
\cR_B/\ol{i_*}\cR_A \ar@<3ex>[l]^-{\ol{j_*}}
\ar@<-3ex>[l]_-{\ol{j_!}}.}
\]
The corollary ~\ref{c:dense} also holds in this case (one should
replace the functor $\pi_C$).
\end{remark}
The following is now quite obviously.
\begin{corollary}
 Let $A$ and $B$ be  finite-dimensional
$k$-algebras with finite global dimension. Assume the root category
$\cR_A$ is not triangulated with the inherited triangle structure.
For any finite dimensional $A\otimes_k B\op$-module $M$, the root
category of the triangular extension of $A$ and $B$ by $M$ is not
triangulated with the inherited triangle structure.
\end{corollary}
\end{appendix}

\def\cprime{$'$}
\providecommand{\bysame}{\leavevmode\hbox
to3em{\hrulefill}\thinspace}
\providecommand{\MR}{\relax\ifhmode\unskip\space\fi MR }
\providecommand{\MRhref}[2]{%
  \href{http://www.ams.org/mathscinet-getitem?mr=#1}{#2}
} \providecommand{\href}[2]{#2}

\end{document}